\documentclass[11pt]{article}
\usepackage{amsthm}
\usepackage{amsmath}
\usepackage{amssymb}
\usepackage{amsfonts}
\usepackage{mathtools,tensor}
\usepackage{bm}
\usepackage{pifont}
\usepackage[title]{appendix}
\usepackage{fix-cm}
\numberwithin{equation}{section}

\usepackage[
left=1in,
right=1in,
top=1in,
bottom=1in,
includeheadfoot,
heightrounded
]{geometry}

\usepackage{indentfirst}

\usepackage[colorlinks,linkcolor=red,anchorcolor=gray,citecolor=green,urlcolor=blue]{hyperref}
\usepackage{url}
\usepackage{cite}
\usepackage{makecell}
\usepackage{changepage}
\usepackage{graphicx}
\usepackage{subcaption} 
\usepackage{color}

\DeclareMathOperator{\supp}{supp}

\newcommand\dbR{\mathbb{R}}

\newcommand{\dbE}{\mathbb{E}}
\newcommand{\dbF}{\mathbb{F}}
\newcommand{\dbH}{\mathbb{H}}

\newcommand{\dbP}{\mathbb{P}}

\newcommand{\ds}{\displaystyle}
\newcommand\ns{\noalign{\smallskip}}

\renewcommand{\=}{\overset{\triangle}{=}}

\renewcommand\a{\alpha}

\newcommand\g{\gamma}

\renewcommand\l{\lambda}

\newcommand\n{\nabla}

\renewcommand\t{\times}

\renewcommand\th{\theta}

\renewcommand\i{\infty}

\newcommand\G{\Gamma}
\newcommand\De{\Delta}

\renewcommand\O{\Omega}

\newcommand\cA{\mathcal{A}}
\newcommand\cB{\mathcal{B}}

\newcommand\cD{\mathcal{D}}

\newcommand\cF{\mathcal{F}}

\newcommand\cM{\mathcal{M}}
\newcommand\cN{\mathcal{N}}

\newcommand\cV{\mathcal{V}}

\newcommand\pa{\partial}
\newcommand\cd{\cdot}

\newcommand\q{\quad}
\newcommand\qq{\qquad}

\def\ba{\begin{array}}
\def\ea{\end{array}}
\newcommand\no{\nonumber}

\newtheorem{theorem}{Theorem}[section]
\newtheorem{lemma}{Lemma}[section]
\newtheorem{definition}{Definition}[section]
\newtheorem{remark}{Remark}[section]
\newtheorem{corollary}{Corollary}[section]
\newtheorem{proposition}{Proposition}[section]

\newtheorem{assumption}{Assumption}[section]
\renewcommand{\proofname}{Proof}

\makeatletter
\renewenvironment{proof}[1][\proofname]
{\par\pushQED{\qed}%
\normalfont\topsep6\p@\@plus6\p@\relax
\noindent\hspace{\parindent}\textbf{#1.} \ignorespaces
}
{\popQED\par\addvspace{\topsep}}
\makeatother

\begin{document}
\title{\bf Unique Continuation Property for Stochastic Wave Equations\thanks{This work is  supported by the NSF of China under grants 12025105 and  12401586. } }
\author{ 
Zhonghua Liao\thanks{School of Mathematics, Sichuan University, Chengdu, 610064, China.  {\small\it
E-mail:} {\small\tt zhonghualiao@yeah.net}. } \q
and \q  Qi L\"{u}\thanks{School of Mathematics, Sichuan University, Chengdu, 610064, China. {\small\it E-mail:} {\small\tt lu@scu.edu.cn}. }}
\date{}
\maketitle

\allowdisplaybreaks

\begin{abstract}
This paper establishes a fundamental and surprising phenomenon in the theory of stochastic wave equations: the restoration of the unique continuation property (UCP) across characteristic hypersurfaces, a property that is known to fail generically in the deterministic setting. We prove that if a solution to a linear stochastic wave equation vanishes on one side of a characteristic surface $\Gamma$, then it must vanish in a full neighborhood of any point on $\Gamma$, provided the stochastic diffusion coefficient is non-degenerate. This result stands in sharp contrast to the classical H\"ormander-type counterexamples for deterministic waves.

Furthermore, we extend the UCP to equations with non-homogeneous stochastic sources  and establish a global unique continuation result from the interior of an arbitrarily narrow characteristic cone. Our proofs rely on a novel stochastic Carleman estimate, where the It\^o diffusion term introduces a crucial positive energy contribution that is absent in deterministic models.  

These findings demonstrate a  qualitative difference between deterministic and stochastic hyperbolic dynamics and open new avenues for control theory and inverse problems in stochastic setting. 
\end{abstract}


\section{Introduction} 

\par  The unique continuation property (UCP) for partial differential equations—the question of whether a solution vanishing in an open set must vanish everywhere—is a cornerstone of PDE theory with profound implications for controllability, inverse problems, and spectral theory. For wave equations, the situation is particularly delicate near  characteristic surfaces, where the classical Holmgren's theorem fails and H\"ormander's celebrated counterexamples show that non-uniqueness is generic, even for equations with analytic coefficients (see Lemma \ref{TheoremHor} below).

In stark contrast, this paper reveals that the introduction of It\^o-type stochastic noise can restore the unique continuation property across characteristic surfaces. This discovery challenges the intuitive boundary between deterministic and stochastic dynamics and highlights a  rigidity imposed by randomness that has no deterministic analogue.

To present the unique continuation results in detail, we first introduce some notations.

Throughout this paper, we assume that $ (\O, \cF, \dbF ,\dbP) $ (with $ \dbF\= \{\cF_t\}_{t\ge 0} $)  on which a one dimensional standard Brownian motion $ \{W(t)\}_{t\ge 0} $ is defined, and $ \dbF $ is the natural filtration generated by $ W(\cd) $. 
\par Let $H$ be a Fr\'echet space. We introduce the following spaces of $H$-valued stochastic processes:
\begin{itemize}
\item $L^{2}_{\mathbb{F}}(0,T;H)$ denotes the space of all $\mathbb{F}$-adapted processes $X(\cdot)$ with values in $H$ such that $
\mathbb{E} \left\vert X(\cdot)\right\vert_{L^{2}(0,T;H)}^{2} < \infty$.
\vspace{-1mm}

\item $L^{\infty}_{\mathbb{F}}(0,T;H)$ denotes the space of all $\mathbb{F}$-adapted, essentially bounded $H$-valued processes, i.e., those satisfying $
\operatorname*{ess\,sup}\limits_{(t,\omega) \in (0,T) \times \Omega} \left\vert X(t,\omega)\right\vert_H < \infty$. 
\vspace{-1mm}
\item $W^{1,\infty}_{\mathbb{F}}(0,T;H)$ denotes the space of all $\mathbb{F}$-adapted processes $X(\cdot)$ such that both $X(\cdot)$ and its (weak) time derivative $\dot{X}(\cdot)$ belong to $L^{\infty}_{\mathbb{F}}(0,T;H)$.
\vspace{-1mm}
\item $L^{2}_{\mathbb{F}}(\Omega; C([0,T]; H))$ denotes the space of all $\mathbb{F}$-adapted processes $X(\cdot)$ with continuous sample paths in $C([0,T]; H)$ satisfying
$
\mathbb{E} \left\vert X(\cdot)\right\vert_{C([0,T]; H)}^{2} < \infty.
$
\vspace{-1mm}
\item   $L^{2}_{\mathbb{F}}(\Omega; C^1([0,T]; H))$  denotes the space of all all $\mathbb{F}$-adapted processes $X(\cdot)$ with   sample paths in $C^1([0,T]; H)$ satisfying
$
\mathbb{E} \left\vert X(\cdot)\right\vert_{C^1([0,T]; H)}^{2}  < \infty.
$
\end{itemize}

All these spaces are endowed with their canonical quasi-norms.

Let $T>0$ and $G\subset \dbR^n$ be a  domain. We consider the following linear stochastic wave equation:
\begin{equation}\label{main_equation} 
du_t -\De u dt = \left(  a_1 u_t +\bm  a_2 \cd \n u+ a_3 u\right)dt + \left(b_1 u_t+b_2 u + f \right)dW(t), 
\end{equation}
where $   a_1,   a_3 \in  L_\dbF^\i (0,T;L^\i (G)) $, $   \bm a_2   \in L_\dbF^\i (0,T;L^{\i} (G;\dbR^n)) $, and $b_1, b_2 \in W^{1,\i}_{\dbF}(0,T;L^\i (G))$, $f\in L_{\dbF}^2(0,T; L^2(G))$.
\par Put 
\begin{equation}
\dbH_T \= L_\dbF^2 (\O; C([0,T]; H_{loc}^1(G)))\cap L_{\dbF}^2(\Omega; C^1([0,T]; L_{loc}^2(G))).
\end{equation}
\begin{definition}
We call $u\in \dbH_T$ a solution of equation \eqref{main_equation} if for each $t\in [0,T]$, $G'\subset\subset G$ and $\eta\in H_0^1(G')$, it holds that 
\begin{align}
&\int_{G'} u_t(t,x) \eta(x)dx -\int_{G'} u_t(0,x) \eta(x)dx \no\\
& =\int_0^t \int_{G'} \left[-\n u(s,x)\cd \n \eta(x) +\left(  a_1 u_t +\bm  a_2 \cd \n u+ a_3 u\right)\eta(x)\right]dxds\\
&\q\q+\int_0^t \int_{G'}\left(b_1 u_t+b_2 u +f \right)\eta (x)dxdW(s),\q \dbP\mbox{-a.s.}\no
\end{align}
\end{definition}

Fix  $(t_0,x_0)\in (0,T)\t G $ and a neighborhood $Q_0\subset (0,T)\t G$ of $(t_0, x_0) $. Consider a  function  $\rho \in C^2(\dbR^{1+n})$   satisfying $\rho(t_0, x_0) =0$ and $\left(\rho_t(t_0,x_), \n \rho(t_0,x_0)\right)\neq 0$.

\par In this paper, we study the unique continuation property (UCP for short) across the surface $\G\= \{(t,x)\in Q_0 \mid \rho(t,x) = 0\}$, that is, we  investigate whether the vanishing of $u$ on one side of a hypersurface $\Gamma = \{\rho = 0\}$ implies its vanishing in a full neighborhood of a point on $\Gamma$. We have established three novel unique continuation principles.

\smallskip

1.  {\bf Local Unique Continuation Across Characteristic Surfaces} (Theorem \ref{main_theorem_1}).   
Under the non-degeneracy condition $|b_1| \geq c_1 > 0$ for some constant $c_1$, we prove that if a solution $u$ to the stochastic wave equation \eqref{main_equation} vanishes on the side $\{\rho > 0\}$, then it also vanishes in a neighborhood of the point. This result holds even when the interface $\Gamma$ is characteristic--that is, when $\rho_t^2 = |\nabla\rho|^2$ on $\Gamma$--a scenario in which deterministic uniqueness generally fails.

\smallskip

2. {\bf Unique Continuation in the Presence of Stochastic Sources} (Theorem \ref{main_theorem_3}).  
We demonstrate that the unique continuation property remains valid even when a stochastic source term $f$ is present. In fact, the vanishing of $u$  to the stochastic wave equation \eqref{main_equation} forces the stochastic source $f$ to vanish as well--a phenomenon that does not generally occur for deterministic wave equations.

\smallskip

3. {\bf Global Unique Continuation from Characteristic Cones} (Theorem \ref{th1}).   
Through an iterative application of the local result, we prove that if a solution  to the stochastic wave equation \eqref{main_equation} vanishes inside an arbitrarily narrow characteristic cone $\{\alpha t^2 > |x|^2\}$, then it must vanish everywhere in spacetime. This global vanishing property, again, does not generally hold for deterministic wave equations.

\smallskip

The three unique continuation properties for the stochastic wave equation discussed above differ fundamentally from those of the deterministic wave equation. One might suspect that the type of the wave equation has been fundamentally altered by the addition of stochastic perturbation—for instance, that it may have become parabolic. Proposition \ref{prop1}, however, shows that this is not the case. In fact, the stochastic wave equation retains a key characteristic of hyperbolic equations: the solution propagates with finite speed.

\smallskip 

The study of UCP across surfaces, also referred to as the Cauchy uniqueness problem, in partial differential equations (PDEs) has a long and rich history, rooted in foundational results from classical analysis. One of the earliest contributions is the Holmgren's uniqueness theorem (see, e.g. \cite[Theorem 8.6.5]{Hor1}), which establishes that if a solution to a linear PDE with analytic coefficients vanishes in an open set intersecting a non-characteristic surface, then the solution must vanish identically in a neighborhood of that surface. This result, however, relies heavily on the analyticity of coefficients and fails for general smooth (non-analytic) coefficients, as demonstrated by counterexamples in the literature.

The framework for the unique continuation property (UCP) was fundamentally reshaped in the 1930s by T. Carleman, who introduced a novel form of weighted energy estimate--now known as the Carleman estimate. This method enabled the proof of unique continuation for solutions to elliptic equations without relying on analyticity assumptions, instead using carefully constructed weight functions to control the ``energy" of the solution.

Subsequent advances in UCP for PDEs were marked by A. P. Calder\'on's 1958 work (\!\! \cite{APC}), which generalized Carleman's weighted estimates to establish UCP across a surface $\Gamma_1$ provided all bicharacteristic curves are transversal to $\Gamma_1$. Later, L. H\"ormander extended these results in \cite{Hormander}, showing that UCP holds for a surface $\Gamma_2$ if all bicharacteristic curves tangent to $\Gamma_2$ exhibit second-order contact with $\Gamma_2$ and remain confined to one side of it. This geometric criterion is now widely referred to as the pseudo-convexity condition.

Over time, Carleman estimates have evolved into one of the central tools in the study of UCP, as evidenced by extensive subsequent research (e.g., \cite{Hormander, NL, CZ, JLR, XFQLXZ} and the references therein).

From an applied mathematics perspective, the  UCP is also indispensable in control theory and inverse problems. For instance, in control problems, UCP ensures that solutions to PDEs can be uniquely determined from partial observations, enabling the design of effective boundary or internal controls \cite{XFQLXZ}. In inverse problems, UCP guarantees the uniqueness of solutions to state recovery problems \cite{MVK}. These applications underscore the theoretical and practical importance of UCP in PDE analysis. Therefore, our work lies at the intersection of several active fields: the theory of unique continuation for PDEs, stochastic wave propagation, and controllability of SPDEs. The results imply, for instance, that in a random medium, local observations may determine interior states more effectively than in deterministic media--a finding with potential applications to stochastic control and tomography.

Randomness and uncertainty are pervasive in natural and engineered systems, often rendering deterministic models inadequate to capture the complexity of real-world phenomena. Stochastic Partial Differential Equations (SPDEs) emerge as a natural mathematical framework for describing systems in which randomness interacts with spatially extended and temporally evolving dynamics. They offer a powerful language for modeling processes influenced by noise, fluctuations, or incomplete information, thereby bridging the gap between idealized deterministic theory and observable reality.

Consequently, research on SPDEs has grown substantially in recent years. Nevertheless, compared with deterministic PDEs, the study of the unique continuation property (UCP) for SPDEs remains in its early stages. To our knowledge, \cite{XZ} represents the first work addressing UCP for SPDEs, establishing a unique continuation result for stochastic parabolic equations. Following this, several studies have extended UCP analysis to other classes of SPDEs. In particular, we refer readers to \cite{Lu2023,QLZY1, XZ1, QL1, QLZY1,Lu2015,Zhang2010} for UCP related to stochastic wave equations.

However, existing work has largely focused on extending known UCP results from deterministic to stochastic wave equations, without revealing genuinely new phenomena unique to the stochastic context. In contrast, the present paper investigates UCP results that hold for stochastic wave equations but fail in the deterministic case, as outlined earlier.

The key tool is a new  Carleman estimate adapted to the It\^o calculus (Theorem \ref{th3}). Unlike its deterministic counterpart, our estimate captures an additional energy term arising from the quadratic variation of the stochastic derivative $\mathrm{d}v_t$. This term provides a crucial positivity that offsets the degeneracy near characteristic surfaces, making the estimate  stronger in the stochastic setting than in the deterministic one. This reversal of fortunes is the mathematical heart of our results.

The rest of this paper is structured as follows.  
Section \ref{sec-main} presents the main results and explains why their deterministic analogues fail.  
Section \ref{sec-Car} derives a Carleman estimate for stochastic wave equations.  
Sections \ref{sec-proof-th1}--\ref{sec-proof-th3} provide proofs of the local UCP theorems.  
Finally, Section \ref{sec-proof-th4} establishes the global result through an iterative geometric argument.

\section{The main results}\label{sec-main}

In this section,  we introduce the main results of this paper.  To begin with, we denote that for any $\varrho  \in C^2(\dbR^{1+n})$,  
\begin{equation}\label{ii-1}
\cM(\varrho) \=\left(\ba{cc}
\rho_{tt}-\varrho& -\n \rho_t^\top \\
-\n \rho_t & \varrho I_{n\t n} +\left(\rho_{x_jx_k}\right)_{1\le j,k\le n}
\ea \right), 
\end{equation}
where $I_{n\t n}$ is the $n\t n$ unit matrix. 

We   introduce the following  assumption to make the UCP holds.
\begin{assumption}\label{con1}
There exist a positive constant $c_0 $   such that 
\begin{equation}\label{c0}
\rho_{t}(t_0,x_0) \ge c_0,
\end{equation}  
and a  function  $\varrho_1 \in C^2(\dbR^{1+n})$ such that  $\cM(\varrho_1)\ge 0 $ is  non-negative  definite at $(t_0,x_0)$.
\end{assumption}
\begin{theorem}\label{main_theorem_1}
Let $u$ be a solution of equation \eqref{main_equation} with source term $f \equiv 0$.  
Assume that the following conditions are satisfied: 
\begin{enumerate}
\item Assumption \ref{con1} holds;
\item $|b_1(t,x,\omega)| \ge c_1$ for a.e. $(t,x,\omega) \in Q_0 \times \Omega$, for some constant $c_1 > 0$.
\end{enumerate}
Then the following  unique continuation property holds:

If $u \equiv 0$ on the set \vspace{-2mm}  
$$
Q_0 \cap \{(t,x) \mid \rho(t,x) > 0\},\vspace{-2mm} 
$$
there exists a neighbourhood $Q_1$ of the point $(t_0,x_0)$ such that $u \equiv 0$ on $Q_1$.
\end{theorem}
\begin{remark}
In contrast to the well-known pseudo-convexity condition in deterministic PDEs (see e.g., \cite[Chapter 4]{NL}), which requires $\cM(0) > 0$ on the subspace  \vspace{-2mm} 
$$
\big\{(s,\bm y) \in \mathbb{R}^{1+n} \;\big|\; s^2 - |\bm y|^2 = 0,\; s\rho_t(t_0,x_0) - \bm y \cdot \nabla\rho(t_0,x_0) = 0 \big\},\vspace{-2mm} 
$$
Assumption \ref{con1} relaxes the strict positivity to $\cM(\varrho) \ge 0$ by introducing an auxiliary function $\varrho_1$. While neither condition implies the other, Assumption \ref{con1} admits the characteristic case--which the classical pseudo‑convexity condition explicitly excludes (see Corollary \ref{cor1}).

The standard pseudo-convex condition is derived via pseudo-differential operator techniques. In this paper, we employ a more elementary approach to Carleman estimates (though less refined), which reveals a key difference:
\begin{itemize}
\item In the deterministic setting, our method requires $\cM(\varrho_1) > 0$;
\item In the stochastic setting, only $\cM(\varrho_1) \ge 0$ is needed.
\end{itemize}
This phenomenon highlights a fundamental distinction between stochastic and deterministic PDEs. To further relax Assumption \ref{con1} and fully characterize this difference, the development of stochastic pseudo‑differential operators will be essential in future work.
\end{remark}
\par By Theorem \ref{main_theorem_1}, we can   obtain that the UCP still holds across  characteristic hypersurfaces.

\begin{corollary}\label{cor1}
If we replace the Assumption \ref{con1} in Theorem \ref{main_theorem_1} by the following:
\par There exists a positive constant $c_0 $ such that $	\rho_t(t_0,x_0) \ge c_0$,   and  \vspace{-2mm} 
\begin{equation}\label{rho-2}
\rho_t^2-|\n \rho|^2 = 0 \mbox{ on }\G\= \{(t,x)\in Q_0|~\rho(t,x)=0\},\vspace{-2mm} 
\end{equation}
The UCP still holds:
\par If $u \equiv 0$ on $Q_0 \cap \{(t,x) \mid \rho(t,x) > 0\}$, there exists a neighborhood $Q_1$ of $(t_0, x_0)$ such that $u \equiv 0$ on $Q_1$.
\end{corollary}

\begin{remark}
Although local unique continuation across a characteristic surface typically fails for deterministic hyperbolic equations (see Section \ref{sec-nonuni} for example), certain global unique continuation results can still be obtained under special geometric or analytic conditions (see e.g. \cite{NL1, AISK, CEKARCDAS}).

For instance, consider the equation  \vspace{-2mm} 
$$
u_{tt} - \Delta u + a u + \mathbf{b} \cdot \nabla u = 0,
\qquad (a,\mathbf{b}) \in C^0(\mathbb{R}^{1+n}; \mathbb{R}^{1+n}),\vspace{-2mm} 
$$
and suppose a solution $u$ vanishes in the exterior region  \vspace{-2mm} 
$$
\mathcal{E} := \{(t,x) \in \mathbb{R}^{1+n} \mid |x| > |t| + 1\},\vspace{-2mm} 
$$
and also vanishes in a neighbourhood of $\partial\mathcal{E}$. Then one can conclude that $u \equiv 0$ in $\mathbb{R}^{1+n}$ (cf. \cite{NL1, AISK}).

Another well‑known example is the wave equation \vspace{-2mm}  
$$
u_{tt} - \Delta u = V u,
\qquad V \in L^{\frac{n+1}{2}}(\mathbb{R}^{1+n}).\vspace{-2mm} 
$$
If a solution $u \in H^{2,p}(\mathbb{R}^{1+n})$ (for some $p \in [1,\infty)$) vanishes on one side of the characteristic hyperplane $\{(t,x)\mid t+x_1=0\}$ (or, after a rotation, on one side of any non‑characteristic direction), then it follows that $u \equiv 0$ on the whole $\mathbb{R}^{1+n}$ (cf. \cite{CEKARCDAS}).

In contrast, the present work establishes local unique continuation across a single characteristic surface by making essential use of the stochastic noise--a phenomenon that, to our knowledge, has not been previously observed. This result highlights how randomness can fundamentally enhance uniqueness properties, distinguishing stochastic partial differential equations from their deterministic counterparts.
\end{remark}

Assumption \ref{con1} in Theorem \ref{main_theorem_1} (under condition \eqref{c0}) can be dropped if the following alternative assumption is satisfied.
\begin{assumption}\label{con2}
There exists a function $\varrho_2 \in C^2(\dbR^{1+n})$ such that
$$
\cM(\varrho_2)- {3| \rho_{t}|\lVert b_1\rVert^2_{L_\dbF^\i(0,T; L^\i(G))}} I_{(1+n)\t (1+n)} \mbox{ is positive definite	at $(t_0, x_0)$.}
$$
\end{assumption}
\begin{theorem}\label{main_theorem_2}
Let $u$ be a solution of equation \eqref{main_equation} with $f \equiv 0$.  
Suppose Assumption \ref{con2} holds. Then the following unique continuation property holds:

If $u \equiv 0$ in the set $
Q_0 \cap \{(t,x) \mid \rho(t,x) > 0\}$, 
there exists a neighborhood $Q_1$ of the point $(t_0,x_0)$ such that $u \equiv 0$ in $Q_1$.
\end{theorem}
\begin{remark}
Unlike the situation where $\rho_{t}(t_0,x_0) \ge c_0$ is assumed, if such a condition is absent, the diffusion terms have a detrimental effect in the energy estimates. Consequently, the matrix $\mathcal{M}(\rho_2)$ must exhibit stronger positivity in order to guarantee the validity of the unique continuation property.
\end{remark}
\begin{remark}
The primary goal of this paper is to  reveal the  phenomenon that ``stochasticity can restore unique continuation across characteristic surfaces." To present this mechanism as clearly as possible, we intentionally focused on the simplest linear model. The methodology we developed--the stochastic Carleman estimate--can, in principle, handle nonlinear terms that are sufficiently regular (e.g., Lipschitz or with controlled growth), as long as they can be managed within the energy estimates. Extending the results to semilinear or quasilinear equations is a natural and important direction for future work. The current paper aims to establish the phenomenon and build the core analytical toolkit.
\end{remark}
\vspace{-2mm} 

We now turn to the unique continuation property for a stochastic wave equation with a non‑homogeneous term in its stochastic diffusion component--another phenomenon that appears only in the stochastic setting.

\begin{assumption}\label{con3}
There exist a constant $c_0 > 0$ and a function $\varrho_3 \in C^2(\mathbb{R}^{1+n})$ such that
\begin{equation}\label{c0-2}
\rho_{t}(t_0,x_0) \ge c_0,
\end{equation} 
and the matrix $\mathcal{M}(\varrho_3)$ is positive definite at $(t_0,x_0)$.
\end{assumption}
\begin{theorem}\label{main_theorem_3}
Let $u$ be a solution of equation \eqref{main_equation} and suppose Assumption \ref{con3} holds. Then the following  unique continuation property is valid:

If $u \equiv 0$ on the set $
Q_0 \cap \{(t,x) \mid \rho(t,x) > 0\}$, 
there exists a neighbourhood $Q_1$ of $(t_0,x_0)$ such that
$$
f \equiv 0 \quad \text{and} \quad u \equiv 0 \quad \text{on } Q_1.
$$
\end{theorem}
\begin{remark}
Theorem \ref{main_theorem_3} establishes that even in the presence of a nonhomogeneous term in the diffusion component, unique continuation can be deduced solely from local information of the solution. This differs from earlier results \cite{Lu2015}, where global data--such as boundary values, normal derivatives, and final values--were required to guarantee the unique continuation property under similar conditions.
\end{remark}

We refer to the above results as ``local" UCP because $u \equiv 0$ holds only in some neighborhood $Q_1$.  
By iteratively applying the local UCP results and the methods developed in this paper, we obtain the following ``global" UCP theorem.

\begin{theorem}\label{th1}
Assume $G = \mathbb{R}^n$ and $\left\vert b_1(t, x, \omega) \right\vert\ge c_1$ for a.e. $(t, x, \omega) \in \mathbb{R}^+ \times \mathbb{R}^n \times \Omega$, where $c_1 > 0$ is a constant. Let $\alpha > 0, s_0, y_0$ be arbitrary, and define
$$
Q_0 = \big\{(t, x) \mid \alpha (t-s_0)^2 - |x-y_0|^2 \ge 0,\ t \ge 0 \big\}.
$$
Then for any solution $u$ of equation \eqref{main_equation}, if $u \equiv 0$ in $Q_0$, it follows that $u \equiv 0$ on $[0, +\infty) \times \mathbb{R}^n$.
\end{theorem}
\begin{remark}
The region $Q_0$ in the hypothesis of Theorem \ref{th1} is unbounded. A natural question is whether the conclusion of Theorem \ref{th1} remains valid if $Q_0$ is replaced by a bounded region. The finite propagation speed of solutions to the stochastic wave equation implies that the conclusion  does not hold in that case.
\end{remark} 
\vspace{-2mm} 
\begin{remark}
While Theorem \ref{th1} proves unique continuation from within a cone--a setting that aligns naturally with the hyperbolic geometry used in our iterative proof--the extension to arbitrary domains poses significant complications. Controlling the temporal evolution of the zero-set resembles the ``finite propagation speed" constraint in deterministic wave equations. A potential pathway toward generalizing these results lies in combining the stochastic Carleman estimate developed here with more sophisticated geometric optics frameworks, particularly for pseudo-convex domains. This remains an open and stimulating direction for further investigation.
\end{remark} 

\section{Carleman estimate for stochastic wave equations}\label{sec-Car}

\par To begin with, we  introduce the following  weight functions:  
\begin{equation}\label{weight}
\begin{cases}
\psi(t,x) \=e^{\g \rho(t,x)},\q \phi (t,x)\=  \psi(t,x)- \mu \left[|x-x_0|^2+(t-t_0)^2\right], \\ \ns\ds  \ell(t,x) \= \l \phi(t,x),  \q \th(t,x) \= e^{\ell(t,x)},
\end{cases}\q (t,x)\in \dbR^{1+n}
\end{equation}
where $\l, \g,  \mu $ are three under determined parameters in Carleman weight function.   Obviously, we have
\begin{equation}\label{Compared_with_two_set}
\left\{(t,x) |~\phi(t,x)\ge    1\right\}\setminus \{(t_0,   x_0)\}\subset \left\{(t,x)|~\rho (t,x)>0\right\},
\end{equation}
which explains the motivation behind our choice of the above form of the weight function. This choice of $\phi$ can be viewed as a ``Convexification'' of $\psi$, ensuring that the Carleman estimate and the  UCP.

We recall the following known pointwise estimate.

\begin{lemma}\label{lem}\cite[Theorem 4.1]{XZ1}
Let $w$ be  an  $H^1(\dbR^n)$ valued  differential process  and $w_t$ be an $L^2(\dbR^n)$ valued It\^o process. Set $\th=e^\ell$, $v=\th w$. Then, for $a.e.$ $x\in \dbR^n$ and $\dbP$-a.s,
\begin{align}\label{pointwise_identity}
&\th\left(-2\ell_t  v_t+2\n\ell\cd \n v+\Psi v\right)\left(dw_t -\De w dt\right)+\n\cd  \cV dt+d\cN \nonumber\\
& =\Big[ \big(\ell_{tt}+\De \ell -\Psi\big) v_t^2+\big(\ell_{tt}-\De\ell+\Psi\big)|\n v|^2+2\sum_{j,k=1}^n\ell_{x_jx_k}v_{x_j}v_{x_k}   -4\n\ell_t\cd\n v v_t\\
& \qq  +\cB v^2+\left(-2\ell_t v_t+2\n \ell\cd\n v+\Psi v\right)^2\Big]dt+\ell_t(d v_t)^2,\nonumber
\end{align} 
where
\begin{equation}\label{VMBA}
\left\{\ba{ll}
\ds \cA \= \ell_t^2-\ell_{tt}-|\n\ell|^2+\De \ell-\Psi,\\
\ns\ds \cN\=\ell_t|\n v|^2+\ell_t   v_t^2-\n\ell\cd\n v  v_t-\Psi v  v_t+\Big(\cA\ell_t+\frac{\Psi_t}{2}\Big)v^2,\\
\ns\ds \cB\=\cA\Psi+\left(\cA\ell_t\right)_t-\n\cd \left(\cA \n \ell\right)+\frac{1}{2}\big(\Psi_{tt}-\De \Psi\big),\\
\ns\ds \cV\=2\left(\n\ell\cd\n v\right)\n v-\n\ell |\n v|^2-2\ell_t\n v   v_t+ \n\ell   v_t^2+\Psi v\n v-\frac{\n\Psi}{2} v^2-\cA v^2\n \ell.
\ea \right.
\end{equation}
and  $(d v_t)^2$ denote the quadratic variation processes of $v_t$.
\end{lemma}

In the sequel, for a positive integer $r$, we denote by $O(\l^{r})$ a function of order $ \l^{r}$ for large $ \l$ (which is  dependent on $ \mu$ and $ \g$).

Based on Lemma \ref{lem}, we  have the following Carleman estimate.

\begin{theorem}\label{th3}   
Let $\varrho \in C^{2}(\mathbb{R}^{1+n})$ be arbitrary, and let $\theta = e^{\ell}$ and $v = \theta w$ be given by \eqref{weight} and Lemma \ref{lem}. Then the following estimate holds:
\begin{align}\label{35}
&\dbE\int_{Q_0}\left[ \th\left(-2\ell_t  v_t+2\n\ell\cd \n v+\Psi v\right)\left(dw_t -\De w dt\right)+\n\cd  \cV dt+d\cN \right]dx\no\\
&\ge \dbE\int_{Q_0} \left[ 2  \l\g^2\psi \left( \rho_t v_t -\n\rho \cd\n v\right)^2 + 2\l\g\psi \left(v_t, \n v^\top \right) \cM(\varrho)\left(v_t, \n v^\top \right)^\top\right]dtdx\no \\
&\q + \dbE\int_{Q_0} \left(- 10\l\mu v_t^2 +2\l\mu|\n v|^2\right)dtdx +\dbE\int_{Q_0} \left[\l^3\cD_2 + \l^3  \cD_3+O(\l^2)\right]v^2dtdx  \\
& \q + \dbE\int_{Q_0} \left(-2\ell_t v_t+2\n \ell\cd\n v+\Psi v\right)^2dtdx+\dbE\int_{Q_0}  \ell_t (dv_t)^2dx,\no 
\end{align}
where $\cM(\varrho)$ is defined in \eqref{ii-1},  $\Psi =  \De\ell  -\ell_{tt} +2\l \g\psi\varrho  +  6\l \mu $,
\begin{equation}\label{D_123}
\cD_1 =  { 4 \mu^2}\left[(t-t_0)^2- |x-x_0|^2\right] + 4\g \mu  \psi \left[(t-t_0)\rho_t+(x-x_0)\cd \n \rho \right],\\
\end{equation}
\begin{equation}
\cD_2  =4\g^3\psi^3\varrho \left(\rho_t^2-|\n \rho|^2\right) +2 \g^3\psi^3\left(\rho_t, \n\rho^\top \right)\cM(\varrho) \left(\rho_t, \n\rho^\top \right)^\top+2\g^4\psi^3 \left(\rho_t^2-|\n \rho|^2\right)^2,
\end{equation}
and\vspace{-2mm}
\begin{align}\label{D_3-1}
\cD_3= &~ 2 \g\psi\varrho\cD_1+ { 6\mu}\left(\psi_t^2-|\n \psi|^2\right)+  {6 \mu}\cD_1- {2 \mu}(t-t_0)\left[ \left(\psi_t^2-|\n \psi|^2\right)+ \cD_1\right]_t  \\
&+\left(\pa_t\cD_{1}\right)\left[\psi_t  - {2 \mu} (t-t_0)\right]- {2 \mu}(x-x_0)\cd \n \left[ \left(\psi_t^2-|\n \psi|^2\right)+ \cD_1\right]\no\\
&+ \n    \cD_1 \cd \left[\n\psi-{2 \mu}(x-x_0)\right].\no 
\end{align}
\end{theorem}

\begin{proof}[Proof of Theorem \ref{th3}]
By direct calculation, we obtain
\begin{equation}\label{phi_1}
\phi_t = \gamma\psi\rho_t - 2\mu(t-t_0), \qquad 
\phi_{x_j} = \gamma\psi\rho_{x_j} - 2\mu\bigl(x_j - x_{0j}\bigr),
\end{equation}
\begin{equation}\label{phi_2}
\phi_{tx_j} = \gamma^2\psi\rho_t\rho_{x_j} + \gamma\psi\rho_{tx_j},
\end{equation}
and
\begin{equation}\label{phi_3}
\phi_{tt} = \gamma^2\psi\rho_t^2 + \gamma\psi\rho_{tt} - 2\mu, \qquad 
\phi_{x_jx_k} = \gamma^2\psi\rho_{x_j}\rho_{x_k} + \gamma\psi\rho_{x_jx_k} - 2\delta_{jk}\mu,
\end{equation}
for $j,k = 1,\dots,n$.

Substituting \eqref{phi_1}--\eqref{phi_3} into \eqref{pointwise_identity} and choosing
$$
\Psi = \Delta\ell - \ell_{tt} + 2\lambda\gamma\psi\varrho + 6\lambda\mu,
$$
we obtain
\begin{equation}\label{v_t}
\bigl(\ell_{tt} + \Delta\ell - \Psi\bigr)v_t^2 
= 2\gamma^2\psi\rho_t^2 + 2\lambda\gamma\psi(\rho_{tt} - \varrho)v_t^2 - 10\lambda\mu v_t^2,
\end{equation}
\begin{align}\label{nv}
&\bigl(\ell_{tt} - \Delta\ell + \Psi\bigr)|\nabla v|^2 
+ 2\sum_{j,k=1}^{n}\ell_{x_jx_k}v_{x_j}v_{x_k} \nonumber\\
&= 2\lambda\gamma\psi\varrho|\nabla v|^2 
+ 2\lambda\psi\sum_{j,k=1}^{n}\bigl(\gamma\rho_{x_jx_k} + \gamma^2\rho_{x_j}\rho_{x_k}\bigr)v_{x_j}v_{x_k}
+ 2\lambda\mu|\nabla v|^2,
\end{align}
and
\begin{equation}\label{nv_t}
-4\nabla\ell_t\cdot\nabla v\,v_t = -4\bigl(\gamma^2\psi\rho_t\nabla\rho + \gamma\psi\nabla\rho_t\bigr)\cdot\nabla v\,v_t.
\end{equation}
Combining \eqref{v_t}, \eqref{nv} and \eqref{nv_t}, we conclude
\begin{align}\label{227}
&\bigl(\ell_{tt} + \Delta\ell - \Psi\bigr)v_t^2 
+ \bigl(\ell_{tt} - \Delta\ell + \Psi\bigr)|\nabla v|^2
+ 2\sum_{j,k=1}^{n}\ell_{x_jx_k}v_{x_j}v_{x_k} 
- 4\nabla\ell_t\cdot\nabla v\,v_t  \\
&= 2\lambda\gamma^2\psi\bigl(\rho_t v_t - \nabla\rho\cdot\nabla v\bigr)^2
+ 2\lambda\gamma\psi\bigl(v_t,\nabla v^{\top}\bigr)\mathcal{M}(\varrho)\bigl(v_t,\nabla v^{\top}\bigr)^{\top} 
- 10\lambda\mu v_t^2 + 2\lambda\mu|\nabla v|^2.\no 
\end{align}
On the other hand,
\begin{equation}\label{A_expr}
\mathcal{A} = \ell_t^2 - \ell_{tt} - |\nabla\ell|^2 + \Delta\ell - \Psi
= \lambda^2\bigl(\psi_t^2 - |\nabla\psi|^2\bigr) + \lambda^2\mathcal{D}_1 + O(\lambda),
\end{equation}
where $\mathcal{D}_1$ is the function defined in \eqref{D_123}, namely
\begin{equation}
\mathcal{D}_1 = 4\mu^2\bigl[(t-t_0)^2 - |x-x_0|^2\bigr]
+ 4\mu\bigl[(t-t_0)\psi_t + (x-x_0)\cdot\nabla\psi\bigr].
\end{equation}
Consequently,
\begin{align}\label{v}
\mathcal{B} &= \mathcal{A}(\Psi + \ell_{tt} - \Delta\ell) + \mathcal{A}_t\ell_t - \nabla\mathcal{A}\cdot\nabla\ell \nonumber\\
&= \lambda^3\bigl(2\gamma\psi\varrho + 6\mu\bigr)
\bigl[\bigl(\psi_t^2 - |\nabla\psi|^2\bigr) + \mathcal{D}_1\bigr]  + \lambda^3\bigl[\bigl(\psi_t^2 - |\nabla\psi|^2\bigr) + \mathcal{D}_1\bigr]_t
\bigl[\psi_t - 2\mu(t-t_0)\bigr] \nonumber\\
&\quad + \lambda^3\nabla\bigl[\bigl(\psi_t^2 - |\nabla\psi|^2\bigr) + \mathcal{D}_1\bigr]
\cdot\bigl[\nabla\psi - 2\mu(x-x_0)\bigr] + O(\lambda^2) \nonumber\\
&= \lambda^3\mathcal{D}_2 + \lambda^3\mathcal{D}_3 + O(\lambda^2),
\end{align}
where $\mathcal{D}_2$ and $\mathcal{D}_3$ are the functions introduced in \eqref{D_123} and \eqref{D_3-1}, respectively. 

By direct computation, we find that
\begin{equation}\label{D_2*}
\mathcal{D}_2 = 2\gamma\psi\varrho\bigl(\psi_t^2 - |\nabla\psi|^2\bigr)
+ \bigl(\psi_t^2 - |\nabla\psi|^2\bigr)_t\psi_t
- \nabla\bigl(\psi_t^2 - |\nabla\psi|^2\bigr)\cdot\nabla\psi,
\end{equation}
and
\begin{align}\label{D_3}
\mathcal{D}_3 &= 2\gamma\psi\varrho\mathcal{D}_1 + 6\mu\bigl(\psi_t^2 - |\nabla\psi|^2\bigr) + 6\mu\mathcal{D}_1  - 2\mu(t-t_0)\Bigl[\bigl(\psi_t^2 - |\nabla\psi|^2\bigr) + \mathcal{D}_1\Bigr]_t \nonumber\\
&\quad + \bigl(\partial_t\mathcal{D}_1\bigr)\bigl[\psi_t - 2\mu(t-t_0)\bigr]   - 2\mu(x-x_0)\cdot\nabla\Bigl[\bigl(\psi_t^2 - |\nabla\psi|^2\bigr) + \mathcal{D}_1\Bigr] \nonumber\\
&\quad + \nabla\mathcal{D}_1\cdot\bigl[\nabla\psi - 2\mu(x-x_0)\bigr].
\end{align}
In \eqref{v}, the terms of order $\mu$ have been separated and collected into $\mathcal{D}_3$.

Finally, rewriting \eqref{D_2*} in the form of \eqref{D_123} by direct calculation gives
\begin{align}\label{D_2-1}
\mathcal{D}_2 &= 2\gamma^3\psi^3\varrho\bigl(\rho_t^2 - |\nabla\rho|^2\bigr)
+ 2\psi_{tt}\psi_t^2 - 4\nabla\psi\cdot\nabla\psi_t\,\psi_t
+ 2\sum_{j,k=1}^{n}\psi_{x_jx_k}\psi_{x_j}\psi_{x_k} \nonumber\\
&= 2\gamma^3\psi^3\varrho\bigl(\rho_t^2 - |\nabla\rho|^2\bigr)  + 2\gamma^3\psi^3\bigl(\rho_t,\nabla\rho^{\top}\bigr)
\begin{pmatrix}
\rho_{tt} & -\nabla\rho_t^{\top}\\[2pt]
-\nabla\rho_t & \bigl(\rho_{x_jx_k}\bigr)_{1\le j,k\le n}
\end{pmatrix}
\bigl(\rho_t,\nabla\rho^{\top}\bigr)^{\top} \\
&\quad + 2\gamma^4\psi^3\bigl(\rho_t^2 - |\nabla\rho|^2\bigr)^2 \nonumber\\
&= 4\gamma^3\psi^3\varrho\bigl(\rho_t^2 - |\nabla\rho|^2\bigr)
+ 2\gamma^3\psi^3\bigl(\rho_t,\nabla\rho^{\top}\bigr)\mathcal{M}(\varrho)\bigl(\rho_t,\nabla\rho^{\top}\bigr)^{\top}
+ 2\gamma^4\psi^3\bigl(\rho_t^2 - |\nabla\rho|^2\bigr)^2.\nonumber
\end{align}
Combining \eqref{227}, \eqref{v}, \eqref{D_3} and \eqref{D_2-1} with Lemma~\ref{lem} yields \eqref{35}.
\end{proof}

\section{Proof of Theorem \ref{main_theorem_1} and Corollary \ref{cor1}}\label{sec-proof-th1}

In this section, we invoke Assumption \ref{con1} to guarantee the positivity of the energy terms on the right-hand side of \eqref{35}. This yields a Carleman estimate, which plays a key role in the proof of Theorem \ref{main_theorem_1}. Accordingly, we set $\varrho = \varrho_{1}$ and write $\mathcal{M}_{1} = \mathcal{M}(\varrho_{1})$.

To specify the region $Q_{1}$ appearing in Theorem \ref{main_theorem_1}, we first determine the parameter $\gamma$ in the weight \eqref{weight}.  
\begin{lemma}\label{Lem-D_2}
There exists a positive constant $\gamma$ such that  
\begin{equation}\label{3.1-d}
\mathcal{D}_{2}(t_{0},x_{0}) \ge 0.
\end{equation}
\end{lemma}
\begin{proof}[Proof of Lemma \ref{Lem-D_2}]
Observe that $\rho(t_{0},x_{0}) = 0$; hence
\begin{equation}
\psi(t_{0},x_{0}) = e^{\gamma\rho(t_{0},x_{0})} = 1.
\end{equation}
By Assumption \ref{con1}, we obtain
\begin{equation}
\mathcal{D}_{2}(t_{0},x_{0}) \ge 
4\gamma^{3}\varrho_{1}(t_{0},x_{0})\bigl(\rho_{t}^{2}(t_{0},x_{0}) - |\nabla\rho(t_{0},x_{0})|^{2}\bigr)
+ 2\gamma^{4}\bigl(\rho_{t}^{2}(t_{0},x_{0}) - |\nabla\rho(t_{0},x_{0})|^{2}\bigr)^{2}.
\end{equation}
If $\rho_{t}^{2}(t_{0},x_{0}) - |\nabla\rho(t_{0},x_{0})|^{2} = 0$, then $\mathcal{D}_{2}(t_{0},x_{0}) \ge 0$ holds trivially.  
If $\rho_{t}^{2}(t_{0},x_{0}) - |\nabla\rho(t_{0},x_{0})|^{2} \neq 0$, the coefficient of $\gamma^{4}$ is positive; therefore, by choosing $\gamma$ sufficiently large, the inequality \eqref{3.1-d} follows.
\end{proof}

Next, we determine the parameter $\mu$. Recall the constants $c_{0}$ from \eqref{c0} and $c_{1}$ introduced in Theorem \ref{main_theorem_1}. Note that
\begin{equation}\label{3.4}
\phi_t(t_{0},x_{0}) = \gamma\rho_t(t_{0},x_{0}) \ge \gamma c_{0}.
\end{equation}
Moreover, $\mathcal{D}_{3}(t,x;\mu) = O(\mu)$ as $\mu \to 0$. Combining this asymptotic property with \eqref{3.1-d} and \eqref{3.4}, we can choose $\mu$ sufficiently small so that
\begin{equation}\label{mu_1}
\frac{1}{2}c_{1}^{2}\,\phi_t(t_{0},x_{0})^{3} + \mathcal{D}_{2}(t_{0},x_{0}) + \mathcal{D}_{3}(t_{0},x_{0}) \ge \frac{\gamma^{3}c_{0}^{3}c_{1}^{2}}{3},
\end{equation}
and
\begin{equation}\label{3.6}
\frac{1}{2}\phi_t(t_{0},x_{0})\,c_{1}^{2} - 11\mu \ge \frac{\gamma c_{0}c_{1}^{2}}{3}.
\end{equation}
Having fixed $\gamma$ and $\mu$ as above, we now define the set $Q_{1}$. By \eqref{mu_1}, \eqref{3.6} and Assumption \ref{con1}, there exists a neighborhood $\widetilde{Q}_{1} \subset Q_{0}$ of $(t_{0},x_{0})$ such that, for every $(t,x) \in \widetilde{Q}_{1}$,
\begin{equation}\label{D_1}
\frac{1}{2}\phi_t^{3}c_{1}^{2} + \mathcal{D}_{2} + \mathcal{D}_{3} \ge \frac{\gamma^{3}c_{0}^{3}c_{1}^{2}}{4},
\end{equation}
\begin{equation}\label{3.8}
\frac{1}{2}\phi_t c_{1}^{2} - 11\mu \ge \frac{\gamma c_{0}c_{1}^{2}}{4},
\end{equation}
and
\begin{equation}\label{M_1}
2\gamma\psi\bigl(v_t,\nabla v^{\top}\bigr)\mathcal{M}_{1}\bigl(v_t,\nabla v^{\top}\bigr)^{\top} \ge -\mu\bigl(v_t^{2}+|\nabla v|^{2}\bigr).
\end{equation}
Let $c_{2} < 1$ be the smallest constant satisfying
\begin{equation}
\bigl\{(t,x) \mid \phi(t,x) > c_{2}\bigr\} \subset \widetilde{Q}_{1} \cup \{(t,x) \mid \rho(t,x) > 0\}.
\end{equation}
Finally, we define $Q_{1}$ by
\begin{equation}\label{Q_1}
Q_{1} = \bigl[Q_{0} \cap \{(t,x) \mid \rho(t,x) > 0\}\bigr] \cup \bigl\{(t,x) \mid \phi(t,x) > c_{2} \text{ and } \rho(t,x) \le 0\bigr\}.
\end{equation}

To begin the proof of Theorem \ref{main_theorem_1}, we introduce the following cut-off function.  
For an arbitrary small positive constant $\varepsilon > 0$, define a smooth function $\chi : \mathbb{R}^{1+n} \to \mathbb{R}$ such that
\begin{equation}\label{chi}
\chi(t,x) = 
\begin{cases}
1, & \text{if } \phi(t,x) > c_{2} + \varepsilon,\\[4pt]
0, & \text{if } \phi(t,x) < c_{2}.
\end{cases}
\end{equation}
Set $w \= \chi u$. A direct computation gives
\begin{align*}
\begin{cases}\ds
w_t = \chi_t u + \chi u_t,  \\
\ns\ds d w_t = \bigl(\chi_{tt}u + 2\chi_t u_t\bigr)dt + \chi\,du_t, \\
\ns\ds \nabla w = u\nabla\chi + \chi\nabla u,  \\
\ns\ds \Delta w = \Delta\chi\,u + 2\nabla\chi\cdot\nabla u + \chi\Delta u. 
\end{cases}
\end{align*}
Consequently,
\begin{align}\label{w_equation}
&dw_t - \Delta w\,dt \nonumber\\
&= \chi\bigl(du_t - \Delta u\,dt\bigr) 
+ \bigl(\chi_{tt}u + 2\chi_t u_t\bigr)dt 
- \bigl(2\nabla\chi\cdot\nabla u + \Delta\chi\,u\bigr)dt \\
&= \bigl(a_1 w_t + \bm{a}_2\cdot\nabla w + a_3 w\bigr)dt   + \bigl[-a_1\chi_t u - (\bm{a}_2\cdot\nabla\chi)u 
+ \chi_{tt}u + 2\chi_t u_t - 2\nabla\chi\cdot\nabla u 
+ \Delta\chi u\bigr]dt \nonumber  \\
&\quad + \bigl(b_1 w_t + b_2 w\bigr)dW(t) 
- b_1\chi_t u\,dW(t).\nonumber
\end{align}

Now let $v \= \theta w = e^{\ell}w$; then $w = e^{-\ell}v$. Computing derivatives yields
\begin{align*}
\begin{cases}\ds w_t = e^{-\ell}(-\ell_t v + v_t),   \\
\ns\ds dw_t = e^{-\ell}\bigl(\ell_t^2 v - 2\ell_t v_t - \ell_{tt}v\bigr)dt + e^{-\ell}dv_t,  \\
\ns\ds \nabla w = e^{-\ell}\bigl(-\nabla\ell\,v + \nabla v\bigr),   \\
\ns\ds \Delta w = e^{-\ell}\bigl(|\nabla\ell|^2 v - 2\nabla\ell\cdot\nabla v - \Delta\ell\,v + \Delta v\bigr).  
\end{cases}
\end{align*}
Hence,
\begin{align}\label{w-e}
\theta\bigl(dw_t - \Delta w\,dt\bigr)
&= dv_t - \Delta v\,dt 
+ \bigl[\bigl(\ell_t^2 - |\nabla\ell|^2\bigr)v 
- \bigl(\ell_{tt} - \Delta\ell\bigr)v 
- 2\ell_t v_t + 2\nabla\ell\cdot\nabla v\bigr]dt \nonumber \\
&= \bigl[\bigl(a_3 - a_1\ell_t - \bm{a}_2\cdot\nabla\ell\bigr)v 
+ a_1 v_t + \bm{a}_2\cdot\nabla v\bigr]dt \nonumber \\
&\quad + \theta\bigl[-a_1\chi_t u - (\bm{a}_2\cdot\nabla\chi)u 
+ \chi_{tt}u + 2\chi_t u_t 
- 2\nabla\chi\cdot\nabla u + \Delta\chi u\bigr]dt \nonumber \\
&\quad + \bigl[\bigl(b_2 - b_1\ell_t\bigr)v + b_1 v_t\bigr]dW(t) 
- \theta b_1\chi_t u\,dW(t).
\end{align}

The following Carleman estimate holds.
\begin{theorem}\label{th3-1}
Suppose the parameters $\gamma,\mu$ are fixed according to Lemma \ref{Lem-D_2} and relations \eqref{mu_1}, \eqref{3.6}. For any $\varepsilon>0$ there exist constants $\lambda_0>0$ and $C>0$ such that for every solution $w$ of \eqref{w-e} and all $\lambda\ge\lambda_0$,
\begin{align}\label{35-1}
&\mathbb{E}\int_{\widetilde{Q_1}} \theta^2\bigl(\lambda^3 w^2+\lambda w_t^2\bigr)\,dt\,dx \notag \\
&\le C\,\mathbb{E}\int_{\widetilde{Q_1}} \theta^2\Big[-a_1\chi_t u-(\boldsymbol a_2\cdot\nabla\chi)u
+\chi_{tt}u+2\chi_t u_t-2\nabla\chi\cdot\nabla u+\Delta\chi u\Big]^2\,dt\,dx \notag \\
&\quad + C\,\mathbb{E}\int_{\widetilde{Q_1}} \theta^2\bigl(b_1\chi_t u\bigr)^2\,dt\,dx .
\end{align}
\end{theorem}

\begin{proof}[Proof of Theorem \ref{main_theorem_1} from Theorem \ref{th3-1}]
From the definition of $\chi$ in \eqref{chi} we obtain
$$
\lambda^3 e^{\lambda (c_2+2\varepsilon)}\,
\mathbb{E}\int_{\widetilde{Q}_{1\varepsilon}} u^2\,dt\,dx
\le C e^{\lambda(c_2+\varepsilon)}\,
\mathbb{E}\int_{\widetilde{Q_1}} \bigl(u^2+u_t^2+|\nabla u|^2\bigr)\,dt\,dx,
$$
where
$$
\widetilde{Q}_{1\varepsilon}\coloneqq
\bigl\{(t,x)\mid \phi(t,x)>c_2+2\varepsilon\;\text{and}\;\rho(t,x)<0\bigr\}.
$$
Letting $\lambda\to+\infty$ forces $u\equiv 0$ on $\widetilde{Q}_{1\varepsilon}$. Sending $\varepsilon\to 0$ then completes the proof.
\end{proof}

\begin{proof}[Proof of Theorem \ref{th3-1}]
Substituting \eqref{w-e} into \eqref{35} yields\vspace{-2mm}
\begin{align}\label{316}
&\mathbb{E}\int_{\widetilde{Q_1}}\!\Big[2\lambda\gamma\psi\bigl(v_t,\nabla v^\top\bigr)\mathcal{M}_1\bigl(v_t,\nabla v^\top\bigr)^\top
-\frac{10\lambda}{\mu}v_t^2+\frac{2\lambda}{\mu}|\nabla v|^2
+O(1)\bigl(v_t^2+|\nabla v|^2\bigr)\Big]dt\,dx \notag \\
&+\mathbb{E}\int_{\widetilde{Q_1}}\bigl[\lambda^3\mathcal{D}_2+\lambda^3\mathcal{D}_3+O(\lambda^2)\bigr]v^2\,dt\,dx
+\mathbb{E}\int_{\widetilde{Q_1}}\ell_t(dv_t)^2\,dx \\
&\le \frac12\,\mathbb{E}\int_{\widetilde{Q_1}}\!\theta^2\Big[-a_1\chi_t u-(\boldsymbol a_2\cdot\nabla\chi)u
+\chi_{tt}u+2\chi_t u_t-2\nabla\chi\cdot\nabla u+\Delta\chi u\Big]^2\,dt\,dx .\no 
\end{align}
Using the inequality \eqref{M_1} in \eqref{316} gives\vspace{-2mm}
\begin{align}\label{Car}
&\mathbb{E}\int_{\widetilde{Q_1}}\!\Big[-\frac{11\lambda}{\mu}v_t^2+\frac{\lambda}{\mu}|\nabla v|^2
+O(1)\bigl(v_t^2+|\nabla v|^2\bigr)\Big]dt\,dx \notag \\
&+\mathbb{E}\int_{\widetilde{Q_1}}\bigl[\lambda^3\mathcal{D}_2+\lambda^3\mathcal{D}_3+O(\lambda^2)\bigr]v^2\,dt\,dx
+\mathbb{E}\int_{\widetilde{Q_1}}\ell_t(dv_t)^2\,dx \\
&\le \frac12\,\mathbb{E}\int_{\widetilde{Q_1}}\!\theta^2\Big[-a_1\chi_t u-(\boldsymbol a_2\cdot\nabla\chi)u
+\chi_{tt}u+2\chi_t u_t-2\nabla\chi\cdot\nabla u+\Delta\chi u\Big]^2\,dt\,dx . \no 
\end{align}
We now employ the term $\ell_t(dv_t)^2$ from the pointwise identity \eqref{pointwise_identity} to produce positivity on the left‑hand side. Observe that\vspace{-2mm}
\begin{align}\label{dv_t-1}
\ell_t(dv_t)^2 = \lambda\phi_t\Big[\bigl(b_2-b_1\ell_t\bigr)v+b_1 v_t+\theta\bigl(-b_1\chi_t u\bigr)\Big]^2\,dt.
\end{align}
Hence,\vspace{-2mm}
\begin{align}\label{dv_t}
&\lambda\phi_t\Big[\bigl(b_2-b_1\ell_t\bigr)v+b_1 v_t+\theta\bigl(-b_1\chi_t u\bigr)\Big]^2 \notag \\
&\ge \frac12\lambda\phi_t\Big[\bigl(b_2-b_1\ell_t\bigr)v+b_1 v_t\Big]^2
-\theta^2\bigl(b_1\chi_t\bigr)^2 u^2   \\
&\ge \frac12\lambda\phi_t\,b_1^2 v_t^2 + \frac12\lambda^3 b_1^2\phi_t^3 v^2
- O(\lambda^2)v^2+\lambda\phi_t(b_2-b_1\lambda\phi_t)b_1 v v_t
-\theta^2\bigl(b_1\chi_t\bigr)^2 u^2  \no  \\
&= \frac12\lambda\phi_t c_1^2 v_t^2 + \frac12\lambda^3\phi_t^3 c_1^2 v^2
+\Bigl[\frac12\lambda\phi_t(b_2-b_1\lambda\phi_t)b_1 v^2\Bigr]_t
- O(\lambda^2)v^2
-\theta^2\bigl(b_1\chi_t\bigr)^2 u^2 .\no 
\end{align}
Inserting \eqref{dv_t} into \eqref{Car} we obtain\vspace{-2mm}
\begin{align}\label{Car1}
&\mathbb{E}\int_{\widetilde{Q_1}}\!\Big[\frac12\lambda\phi_t c_1^2 v_t^2
-\frac{11\lambda}{\mu}v_t^2+\frac{\lambda}{\mu}|\nabla v|^2
+O(1)\bigl(v_t^2+|\nabla v|^2\bigr)\Big]dt\,dx \notag \\
&+\mathbb{E}\int_{\widetilde{Q_1}}\!\Big[\frac12\lambda^3\phi_t^3 c_1^2
+\lambda^3\mathcal{D}_2+\lambda^3\mathcal{D}_3+O(\lambda^2)\Big]v^2\,dt\,dx \\
&\le \frac12\,\mathbb{E}\int_{\widetilde{Q_1}}\!\theta^2\Big[-a_1\chi_t u-(\boldsymbol a_2\cdot\nabla\chi)u
+\chi_{tt}u+2\chi_t u_t-2\nabla\chi\cdot\nabla u+\Delta\chi u\Big]^2\,dt\,dx \notag \\
&\quad +\mathbb{E}\int_{\widetilde{Q_1}}\theta^2\bigl(b_1\chi_t\bigr)^2 u^2\,dt\,dx .\no 
\end{align}
Finally, using \eqref{D_1} and \eqref{3.8},
\begin{align}\label{Car1-1}
&\mathbb{E}\int_{\widetilde{Q_1}}\!\Big[\frac{\gamma c_0 c_1^2}{4}\lambda v_t^2
+\frac{\lambda}{\mu}|\nabla v|^2+O(1)\bigl(v_t^2+|\nabla v|^2\bigr)\Big]dt dx  +\mathbb{E}\int_{\widetilde{Q_1}}\!\Big(\frac{\gamma^3c_0^3 c_1^2}{4}\lambda^3+O(\lambda^2)\Big)v^2 dt dx \notag\\
&\le \frac12\,\mathbb{E}\int_{\widetilde{Q_1}}\!\theta^2\big[-a_1\chi_t u-(\boldsymbol a_2\cdot\nabla\chi)u
+\chi_{tt}u+2\chi_t u_t-2\nabla\chi\cdot\nabla u+\Delta\chi u\big]^2\,dt\,dx  \\
&\quad +\mathbb{E}\int_{\widetilde{Q_1}}\theta^2\bigl(b_1\chi_t\bigr)^2 u^2\,dt\,dx .\notag
\end{align}
This directly implies the estimate \eqref{35-1}.
\end{proof}

\begin{proof}[Proof of Corollary \ref{cor1}]
We verify that Assumption \ref{con1} remains valid in the present setting.

Because $\rho_t(t_0,x_0)\ge c_0$, the inverse function theorem allows us to write $\rho$ locally near $(t_0,x_0)$ as\vspace{-2mm}
$$
\rho(t,x)=t-g(x)\vspace{-2mm}
$$
for some $C^2$ function $g$. On the surface $\Gamma$ near $(t_0,x_0)$ we have $|\nabla g|^2=1$; consequently both $(1,-\nabla g)$ and $\bigl(0,\nabla(|\nabla g|^2)\bigr)$ are normal vectors to $\Gamma$ near $(t_0,x_0)$. This implies
\begin{equation}\label{a1}
\nabla\bigl(|\nabla g|^2\bigr)=0 \quad\text{at }(t_0,x_0).
\end{equation}
We may replace $\rho$ by a different representation that leaves $\Gamma$ unchanged locally:\vspace{-2mm}
$$
\rho(t,x)=e^{\tau t}-e^{\tau g(x)},\vspace{-2mm}
$$
where the constant $\tau>0$ is chosen so that the matrix\vspace{-2mm}
\begin{equation}\label{a2}
I_{n\times n}-\tau^{-1}\bigl(g_{x_jx_k}(t_0,x_0)\bigr)_{1\le j,k\le n}\vspace{-2mm}
\end{equation}
is positive definite.

Take $\varrho=\rho_{tt}$ in \eqref{ii-1}. To check $\mathcal{M}(\rho_{tt})\ge 0$, it suffices to show that the matrix
$$
\widetilde{\mathcal{M}}\=\rho_{tt}\,I_{n\times n}+\bigl(\rho_{x_jx_k}\bigr)_{1\le j,k\le n}
$$
is non‑negative at $(t_0,x_0)$. Direct computation gives\vspace{-2mm}
$$
\rho_{tt}=\tau^{2}e^{\tau t},\qquad
\rho_{x_jx_k}=-\tau^{2}e^{\tau g}g_{x_j}g_{x_k}-\tau e^{\tau g}g_{x_jx_k}.\vspace{-2mm}
$$
Hence at $(t_0,x_0)$\vspace{-2mm}
$$
\widetilde{\mathcal{M}}(t_0,x_0)=
\tau^{2}e^{\tau t_0}\Bigl[
I_{n\times n}
-\nabla g(t_0,x_0)\nabla g(t_0,x_0)^\top
-\tau^{-1}\bigl(g_{x_jx_k}(t_0,x_0)\bigr)_{1\le j,k\le n}
\Bigr].\vspace{-2mm}
$$
Let $V$ be the orthogonal complement of $\operatorname{span}\{\nabla g(t_0,x_0)\}$ in $\mathbb{R}^n$. Any vector $\bm x\in\mathbb{R}^n$ can be decomposed as $\bm x=\bm y+k\nabla g(t_0,x_0)$ with $\bm y\in V$ and $k\in\mathbb{R}$, satisfying $\bm y\cdot\nabla g(t_0,x_0)=0$. Using \eqref{a1}, \eqref{a2} and this decomposition we obtain
\begin{align}\label{a4}
&(\bm y+k\nabla g(t_0,x_0))^\top\widetilde{\mathcal{M}}(t_0,x_0)(\bm y+k\nabla g(t_0,x_0)) \notag\\
&  =\tau^{2}e^{\tau t_0}\,\bm y^\top\Bigl[
I_{n\times n}-\tau^{-1}\bigl(g_{x_jx_k}(t_0,x_0)\bigr)_{1\le j,k\le n}
\Bigr]\bm y\ge 0 .
\end{align}
Thus Assumption \ref{con1} is satisfied, and the unique continuation property follows from Theorem \ref{main_theorem_1}.
\end{proof} 

\section{Proof  of Theorem \ref{main_theorem_2}}\label{sec-proof-th2}

In this section, we employ Assumption \ref{con2} to ensure positivity of the energy terms on the right‑hand side of \eqref{35}, thereby obtaining a Carleman estimate suitable for proving Theorem \ref{main_theorem_2}. Throughout this section we set $\varrho = \varrho_2$ and write $\mathcal{M}_2 = \mathcal{M}(\varrho_2)$.

Following the same line as in Lemma \ref{Lem-D_2}, Assumption \ref{con2} together with the relation $\phi_t(t_0,x_0) = \gamma \rho_t(t_0,x_0)$ allows us to choose $\gamma$ so that
\begin{equation}\label{4.2}
\mathcal{D}_2(t_0,x_0)-3|\phi_t(t_0,x_0)|^3\|b_1\|_{L_{\mathbb{F}}^\infty(0,T;L^\infty(G))}^2 > 0.
\end{equation}
Next we fix the parameter $\mu$ in a way similar to \eqref{mu_1} and \eqref{3.6}. Choose $\mu$ and a small constant $\delta > 0$ such that\vspace{-2mm}
\begin{equation}
\mathcal{D}_2(t_0,x_0)+\mathcal{D}_3(t_0,x_0)-3|\phi_t(t_0,x_0)|^3\|b_1\|_{L_{\mathbb{F}}^\infty(0,T;L^\infty(G))}^2 > \delta,\vspace{-2mm}
\end{equation}
and\vspace{-2mm}
\begin{equation}\label{4.4}
2\gamma\psi(t_0,x_0)\bigl(v_t,\nabla v^\top\bigr)\mathcal{M}_2(t_0,x_0)\bigl(v_t,\nabla v^\top\bigr)^\top
-3|\phi_t(t_0,x_0)|\,\|b_1\|_{L_{\mathbb{F}}^\infty(0,T;L^\infty(G))}^2\,v_t^2
-\frac{10}{\mu}v_t^2 > \delta v_t^2.\vspace{-2mm}
\end{equation}
Once $\gamma$ and $\mu$ are chosen as above, we select a neighbourhood $\widetilde Q_1 \subset Q_0$ of $(t_0,x_0)$ such that for all $(t,x)\in\widetilde Q_1$,
\begin{equation}\label{p1}
\mathcal{D}_2+\mathcal{D}_3-3|\phi_t|^3\|b_1\|_{L_{\mathbb{F}}^\infty(0,T;L^\infty(G))}^2 \ge \delta,
\end{equation}
and
\begin{equation}\label{p}
2\gamma\psi\bigl(v_t,\nabla v^\top\bigr)\mathcal{M}_2\bigl(v_t,\nabla v^\top\bigr)^\top
-3|\phi_t|\,\|b_1\|_{L_{\mathbb{F}}^\infty(0,T;L^\infty(G))}^2\,v_t^2
-\frac{10}{\mu}v_t^2 \ge \delta v_t^2 .
\end{equation}
Let $c_2<1$ be the smallest constant for which
$$
\bigl\{(t,x)\mid \phi(t,x) > c_2\bigr\}
\subset \widetilde Q_1 \cup \{(t,x)\mid \rho(t,x)>0\}.
$$

We now define the neighbourhood $Q_1$ as
\begin{equation}\label{Q_1-1}
Q_1 \=
\big[ Q_0 \cap \{(t,x)\mid \rho(t,x)>0\} \big]
\;\cup\;
\big\{\,(t,x)\mid \phi(t,x) > c_2 \text{ and } \rho(t,x)\le 0\,\big\}.
\end{equation}
As before, introduce a cut‑off function $\chi:\mathbb{R}^{1+n}\to\mathbb{R}$ satisfying
\begin{equation}\label{chi-1}
\chi(t,x)=1 \quad\text{on } \{(t,x)\mid\phi(t,x)>c_2+\varepsilon\},
\qquad
\chi(t,x)=0 \quad\text{on } \{(t,x)\mid\phi(t,x)<c_2\}.
\end{equation}
 Set $w \= \chi u$ and $v \= \th w = e^\ell w$. A computation analogous to \eqref{w-e} gives
\begin{align}\label{w-e-1}
&\theta\bigl(dw_t-\Delta w\,dt\bigr) \notag \\
&= dv_t-\Delta v\,dt
+\big[(\ell_t^2-|\nabla\ell|^2)v-(\ell_{tt}-\Delta\ell)v-2\ell_t v_t+2\nabla\ell\cdot\nabla v\big]dt \notag \\
&= \bigl[(a_3-a_1\ell_t-\boldsymbol a_2\cdot\nabla\ell)v+a_1 v_t+\boldsymbol a_2\cdot\nabla v\bigr]dt \notag \\
&\quad+\theta\big[-a_1\chi_t u-(\boldsymbol a_2\cdot\nabla\chi)u
+\chi_{tt}u+2\chi_t u_t-2\nabla\chi\cdot\nabla u+\Delta\chi u\big]dt \\
&\quad+\bigl[(b_2-b_1\ell_t)v+b_1 v_t\bigr]dW(t)  +\theta\bigl(-b_1\chi_t u\bigr)dW(t).\notag
\end{align}

The proof of Theorem \ref{main_theorem_2} follows immediately from the following Carleman estimate, in the same way as Theorem \ref{main_theorem_1}.
\begin{theorem}\label{th3-2}
Suppose the parameters $\gamma,\mu$ are chosen according to \eqref{4.2}--\eqref{4.4}. For every $\varepsilon>0$ there exist constants $\lambda_0>0$ and $C>0$ such that for any solution $w$ of \eqref{w-e-1} and all $\lambda\ge\lambda_0$,\vspace{-2mm}
\begin{align}\label{35-1-1}
&\mathbb{E}\int_{\widetilde Q_1} \theta^2 \bigl(\lambda^3 w^2+\lambda w_t^2\bigr)\,dt\,dx \notag \\
&\le C\,\mathbb{E}\int_{\widetilde Q_1} \theta^2\big[-a_1\chi_t u-(\boldsymbol a_2\cdot\nabla\chi)u
+\chi_{tt}u+2\chi_t u_t-2\nabla\chi\cdot\nabla u+\Delta\chi u\big]^2 dt dx  \\
&\quad +C\lambda\,\mathbb{E}\int_{\widetilde Q_1} \theta^2\bigl(b_1\chi_t u\bigr)^2\,dt\,dx.\notag
\end{align}
\end{theorem}
\begin{proof}[Proof of Theorem \ref{th3-2}]
From \eqref{35} we obtain
\begin{align}\label{Car-1}
&\mathbb{E}\int_{Q_1}\!\Big[2\lambda\gamma\psi\bigl(v_t,\nabla v^\top\bigr)\mathcal{M}_2\bigl(v_t,\nabla v^\top\bigr)^\top
-\frac{10\lambda}{\mu}v_t^2+\frac{2\lambda}{\mu}|\nabla v|^2
+O(1)\bigl(v_t^2+|\nabla v|^2\bigr)\Big]dt\,dx \notag \\
&+\mathbb{E}\int_{Q_1}\lambda^3\bigl(\mathcal{D}_2+\mathcal{D}_3+O(\lambda^{-2})\bigr)v^2\,dt\,dx  +\mathbb{E}\int_{Q_1}\ell_t(dv_t)^2\,dx \\
&\le \frac12\,\mathbb{E}\int_{Q_1}\!\theta^2\big[-a_1\chi_t u-(\boldsymbol a_2\cdot\nabla\chi)u
+\chi_{tt}u+2\chi_t u_t-2\nabla\chi\cdot\nabla u+\Delta\chi u\big]^2 dt dx.\notag
\end{align}
Now we estimate the term $\ell_t(dv_t)^2$. Using \eqref{dv_t-1} we have\vspace{-2mm}
\begin{align}\label{dv_t-2}
&\lambda\phi_t\big[\bigl(b_2-b_1\ell_t\bigr)v+b_1 v_t+\theta\bigl(-b_1\chi_t u\bigr)\big]^2 \notag \\
&\ge -3\lambda|\phi_t|\big[\bigl(b_2-b_1\ell_t\bigr)^2v^2+b_1^2 v_t^2+\theta^2\bigl(b_1\chi_t\bigr)^2u^2\big] \\
&\ge -3\lambda|\phi_t|\,b_1^2 v_t^2-3\lambda^3 b_1^2|\phi_t|^3 v^2-O(\lambda^2)v^2
-3\lambda|\phi_t|\,\theta^2\bigl(b_1\chi_t\bigr)^2u^2 . \notag
\end{align}
Combining \eqref{Car-1} with \eqref{p1}, \eqref{p} and \eqref{dv_t-2} yields the desired estimate \eqref{35-1-1}.
\end{proof}

\section{Proof Theorem \ref{main_theorem_3}}\label{sec-proof-th3}

In this section we present the proof of Theorem \ref{main_theorem_3}. Here we take $\varrho = \rho_3$ and write $\mathcal{M}_3 = \mathcal{M}(\varrho_3)$.

Following the same argument as in Lemma \ref{Lem-D_2}, Assumption \ref{con2} allows us to choose $\gamma$ such that
\begin{equation}\label{4.2-1}
\mathcal{D}_2(t_0,x_0) > 0.
\end{equation}
Next we fix the parameter $\mu$. Choose $\mu$ and a small constant $\delta > 0$ satisfying\vspace{-2mm}
$$
\mathcal{D}_2(t_0,x_0)+\mathcal{D}_3(t_0,x_0) > \delta,\vspace{-2mm}
$$
and\vspace{-2mm}
\begin{equation}\label{4.4-1}
2\gamma\psi(t_0,x_0)\bigl(v_t,\nabla v^\top\bigr)\mathcal{M}_3(t_0,x_0)\bigl(v_t,\nabla v^\top\bigr)^\top
- \frac{10}{\mu}v_t^2 > \delta v_t^2.\vspace{-2mm}
\end{equation}
Once $\gamma$ and $\mu$ are chosen in this way, we can pick a neighbourhood $\widetilde Q_1\subset Q_0$ of $(t_0,x_0)$ such that for all $(t,x)\in\widetilde Q_1$,\vspace{-2mm}
\begin{gather}
\label{4.5-1}
\mathcal{D}_2+\mathcal{D}_3 \ge \delta,\\[4pt]
\label{4.6-1}
2\gamma\psi\bigl(v_t,\nabla v^\top\bigr)\mathcal{M}_3\bigl(v_t,\nabla v^\top\bigr)^\top
- \frac{10}{\mu}v_t^2 \ge \delta v_t^2,\vspace{-2mm}
\end{gather}
and, recalling that $\phi_t(t_0,x_0)=\gamma\rho_t(t_0,x_0)\ge\gamma c_0$,\vspace{-2mm}
\begin{equation}\label{5.6}
\phi_t(t,x) \ge \delta.\vspace{-2mm}
\end{equation}
Let $c_2<1$ be the largest constant for which
$$
\bigl\{(t,x)\mid \phi(t,x) > c_2\bigr\}
\subset \widetilde Q_1 \cup \{(t,x)\mid \rho(t,x)>0\}.
$$

We then define the neighbourhood $Q_1$ as\vspace{-2mm}
\begin{equation}\label{Q_1-1-1}
Q_1 \=
\big[\,Q_0 \cap \{(t,x)\mid \rho(t,x)>0\}\,\big]
\;\cup\;
\big\{\,(t,x)\mid \phi(t,x) > c_2 \text{ and } \rho(t,x)\le 0\,\big\}.\vspace{-2mm}
\end{equation}
Introduce a cut‑off function $\chi:\mathbb{R}^{1+n}\to\mathbb{R}$ by\vspace{-2mm}
\begin{equation}\label{chi-1-1}
\chi(t,x)=1 \quad\text{on } \big\{(t,x)\mid\phi(t,x)>c_2+\varepsilon\big\},
\qquad
\chi(t,x)=0 \quad\text{on } \big\{(t,x)\mid\phi(t,x)<c_2\big\}.\vspace{-2mm}
\end{equation}
Set $w \= \chi u$ and $v \= \th w = e^\ell w$.  A computation similar to \eqref{w-e} gives
\begin{align}\label{w-e-1-1}
&\theta\bigl(dw_t-\Delta w\,dt\bigr) \notag \\
&= dv_t-\Delta v\,dt
+\big[(\ell_t^2-|\nabla\ell|^2)v-(\ell_{tt}-\Delta\ell)v-2\ell_t v_t+2\nabla\ell\cdot\nabla v\big]dt \notag \\
&= \bigl[(a_3-a_1\ell_t-\boldsymbol a_2\cdot\nabla\ell)v+a_1 v_t+\boldsymbol a_2\cdot\nabla v\bigr]dt \notag \\
&\quad+\theta\big[-a_1\chi_t u-(\boldsymbol a_2\cdot\nabla\chi)u
+\chi_{tt}u+2\chi_t u_t-2\nabla\chi\cdot\nabla u+\Delta\chi u\big]dt \\
&\quad+\bigl[(b_2-b_1\ell_t)v+b_1 v_t+\theta\chi f\bigr]dW(t)  +\theta\bigl(-b_1\chi_t u\bigr)dW(t).\notag
\end{align}

Similar to the proof of Theorem \ref{main_theorem_1}, Theorem \ref{main_theorem_3} follows at once from the following Carleman estimate.
\begin{theorem}\label{th3-3}
Suppose the parameters $\gamma,\mu$ are chosen according to \eqref{4.2-1}–\eqref{4.4-1}. For every $\varepsilon>0$ there exist constants $\lambda_0>0$ and $C>0$ such that for any solution $w$ of \eqref{w-e-1-1} and all $\lambda\ge\lambda_0$,
\begin{align}\label{35-1-1-1}
&\mathbb{E}\int_{\widetilde Q_1} \theta^2\bigl(\lambda^3 w^2+\lambda w_t^2\bigr)\,dt\,dx
+\mathbb{E}\int_{\widetilde Q_1} \bigl[(b_2-b_1\ell_t)v+b_1 v_t+\theta\chi f\bigr]^2\,dt\,dx \notag \\
&\le C\,\mathbb{E}\int_{\widetilde Q_1} \theta^2\Big[-a_1\chi_t u-(\boldsymbol a_2\cdot\nabla\chi)u
+\chi_{tt}u+2\chi_t u_t-2\nabla\chi\cdot\nabla u+\Delta\chi u\Big]^2\,dt\,dx \\
&\quad +C\lambda\,\mathbb{E}\int_{\widetilde Q_1} \theta^2\bigl(b_1\chi_t u\bigr)^2\,dt\,dx.\notag 
\end{align}
\end{theorem}

\begin{proof}[Proof of Theorem \ref{th3-3}]
From \eqref{35} we obtain
\begin{align}\label{Car-1*}
&\mathbb{E}\int_{Q_1}\!\Big[2\lambda\gamma\psi\bigl(v_t,\nabla v^\top\bigr)\mathcal{M}_3\bigl(v_t,\nabla v^\top\bigr)^\top
-\frac{10\lambda}{\mu}v_t^2+\frac{2\lambda}{\mu}|\nabla v|^2
+O(1)\bigl(v_t^2+|\nabla v|^2\bigr)\Big]dt\,dx \notag \\
&+\mathbb{E}\int_{Q_1}\lambda^3\bigl(\mathcal{D}_2+\mathcal{D}_3+O(\lambda^{-2})\bigr)v^2\,dt\,dx +\mathbb{E}\int_{Q_1}\ell_t(dv_t)^2\,dx \\
&\le \frac12\,\mathbb{E}\int_{Q_1}\!\theta^2\Big[-a_1\chi_t u-(\boldsymbol a_2\cdot\nabla\chi)u
+\chi_{tt}u+2\chi_t u_t-2\nabla\chi\cdot\nabla u+\Delta\chi u\Big]^2\,dt\,dx.\notag
\end{align}
Now we estimate the term $\ell_t(dv_t)^2$. Using \eqref{dv_t-1} and the lower bound $\phi_t\ge\delta$ from \eqref{5.6},
\begin{align}\label{dv_t-3}
&\lambda\phi_t\big[\bigl(b_2-b_1\ell_t\bigr)v+b_1 v_t+\theta\bigl(-b_1\chi_t u\bigr)\big]^2 \notag \\
&\ge \frac12\lambda\phi_t\big[\bigl(b_2-b_1\ell_t\bigr)v+b_1 v_t+\theta\chi f\big]^2
-\lambda\phi_t\,\theta^2\bigl(b_1\chi_t\bigr)^2u^2 \notag \\
&\ge \frac{\delta}{2}\big[\bigl(b_2-b_1\ell_t\bigr)v+b_1 v_t+\theta\chi f\big]^2
-\lambda\phi_t\,\theta^2\bigl(b_1\chi_t\bigr)^2u^2 .
\end{align}
Combining \eqref{4.5-1}, \eqref{4.6-1} and \eqref{dv_t-3} with \eqref{Car-1*} yields the desired estimate \eqref{35-1-1-1}.
\end{proof}

\section{Proof of Theorem \ref{th1}}\label{sec-proof-th4}

For simplicity and without loss of generality, we consider only the case $\alpha\in(0,1)$ and take $(s_0,y_0)=(0,\mathbf{0})$.

Fix a point $(t_0,x_0)\in\mathbb{R}^+\times\mathbb{R}^n$ and define the forward cone
$$
Q_0(t_0,x_0)\= \Big\{(t,x)\in[t_0,+\infty)\times\mathbb{R}^n \;\Big|\; 
\alpha(t-t_0)^2-|x-x_0|^2\ge0\Big\}.
$$
Set \vspace{-2mm}
\begin{equation}\label{b14}
c_3\=\max\!\left\{\frac{8(\alpha-2)^2}{c_1^4\alpha},\;\frac{32^2}{2c_1^4\alpha^3}\right\}+1,\vspace{-2mm}
\end{equation}
and choose a point $x_1\in\mathbb{R}^n$ satisfying $
|x_1-x_0|^2=4c_3$. 
Define a second conical set
$$
Q_1(t_0,x_1)\=\Big\{(t,x)\in[t_0,+\infty)\times\mathbb{R}^n \;\Big|\;
\frac{\alpha}{2}(t-t_0)^2-|x-x_1|^2>c_3\Big\}.
$$
\begin{remark}
The constant $c_3$ is chosen precisely so that the Carleman estimate \eqref{b10} remains valid.
The point $x_1$ is taken sufficiently far from $x_0$ that the vertex of the hyperboloid $\partial Q_1(t_0,x_1)$ lies outside $Q_0(t_0,x_0)$.
This geometric separation guarantees that condition \eqref{b2} below holds.
\end{remark}

Now let $x_2 = x_1 + k(x_0-x_1)$ with $k>0$ and let $(t_2,x_2)$ be the unique solution of
$$
\begin{cases}
\displaystyle \alpha(t_2-t_0)^2-|x_2-x_0|^2=0,\\[4pt]
\displaystyle \frac{\alpha}{2}(t_2-t_0)^2-|x_2-x_1|^2=c_3.
\end{cases}
$$
A straightforward computation yields
$$
t_2 = t_0+\sqrt{\frac{4c_3}{\alpha}\Bigl(\sqrt{\frac{3}{2}}-2\Bigr)^2},\qquad
x_2 = x_1+\Bigl(\sqrt{\frac{3}{2}}-1\Bigr)(x_0-x_1).
$$
One readily checks that $t_2$ is the minimal time on the intersection $\partial Q_0(t_0,x_0)\cap\overline{Q_1(t_0,x_1)}$.  
Moreover, for any point $(t_2,x)$ of the form $x = x_1 + \tilde k(x_0-x_1)$ with $|\tilde k|<\sqrt{\frac{3}{2}}-1$, we have
\begin{equation}\label{b2}
(t_2,x)\in Q_1(t_0,x_1)\setminus Q_0(t_0,x_0).
\end{equation}

Observe that the quantities $t_2-t_0$ and $|x_1-x_2|$ depend only on $\alpha$ and $c_3$, and are independent of the base point $(t_0,x_0,x_1)$. To simplify the notation in the iterative argument that follows, we introduce two absolute constants:\vspace{-2mm}
$$
T_0\=  t_2-t_0,\qquad X_0\= |x_1-x_2|.\vspace{-2mm}
$$

\begin{lemma}\label{lemma6.1}
Let $u$ be a solution of equation \eqref{main_equation}. If $u,\;u_t,\;\nabla u$ vanish on the set $\partial Q_0(t_0,x_0)\cap Q_1(t_0,x_1)$, 
then $u\equiv 0$ in $Q_1(t_0,x_1)\setminus Q_0(t_0,x_0)$.
\end{lemma}

Before proving Theorem \ref{th1}, we sketch the geometric idea with the help of the following figures.

\begin{figure}[htbp]
\centering
\begin{subfigure}{0.35\textwidth} 
\includegraphics[width=\textwidth]{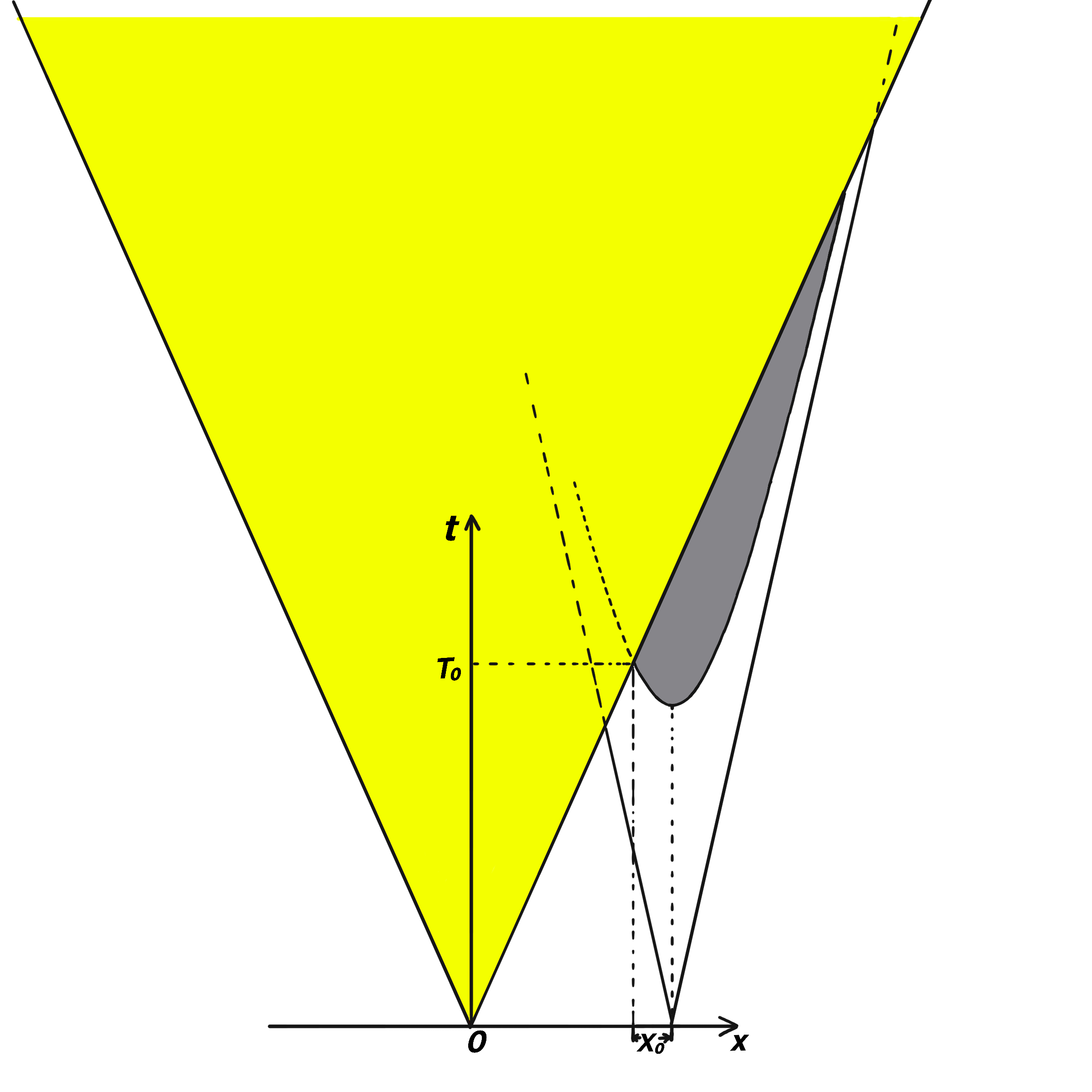} 
\caption{The yellow region indicates the initial domain where $u\equiv 0$. The gray region represents the extended zero zone derived directly from Lemma 6.1.}
\label{fig:1a}
\end{subfigure}
\hfill 
\begin{subfigure}{0.35\textwidth}
\includegraphics[width=\textwidth]{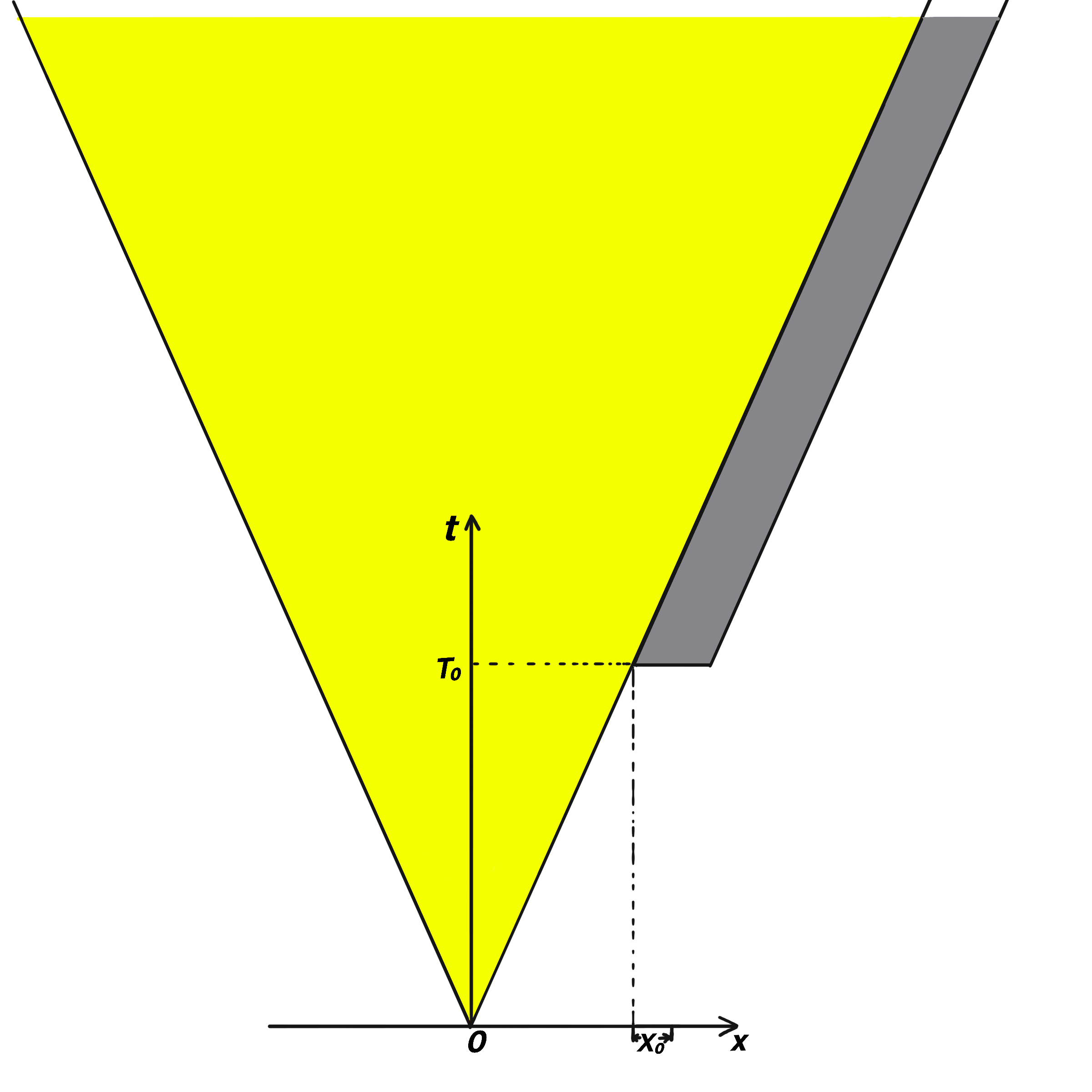} 
\caption{As  $(t_0,x_0)$ varies along the boundary of the initial zero region, Lemma 6.1 ensures $u\equiv 0$ in the expanded gray area.}
\label{fig:1b}
\end{subfigure}

\begin{subfigure}{0.35\textwidth}
\includegraphics[width=\textwidth]{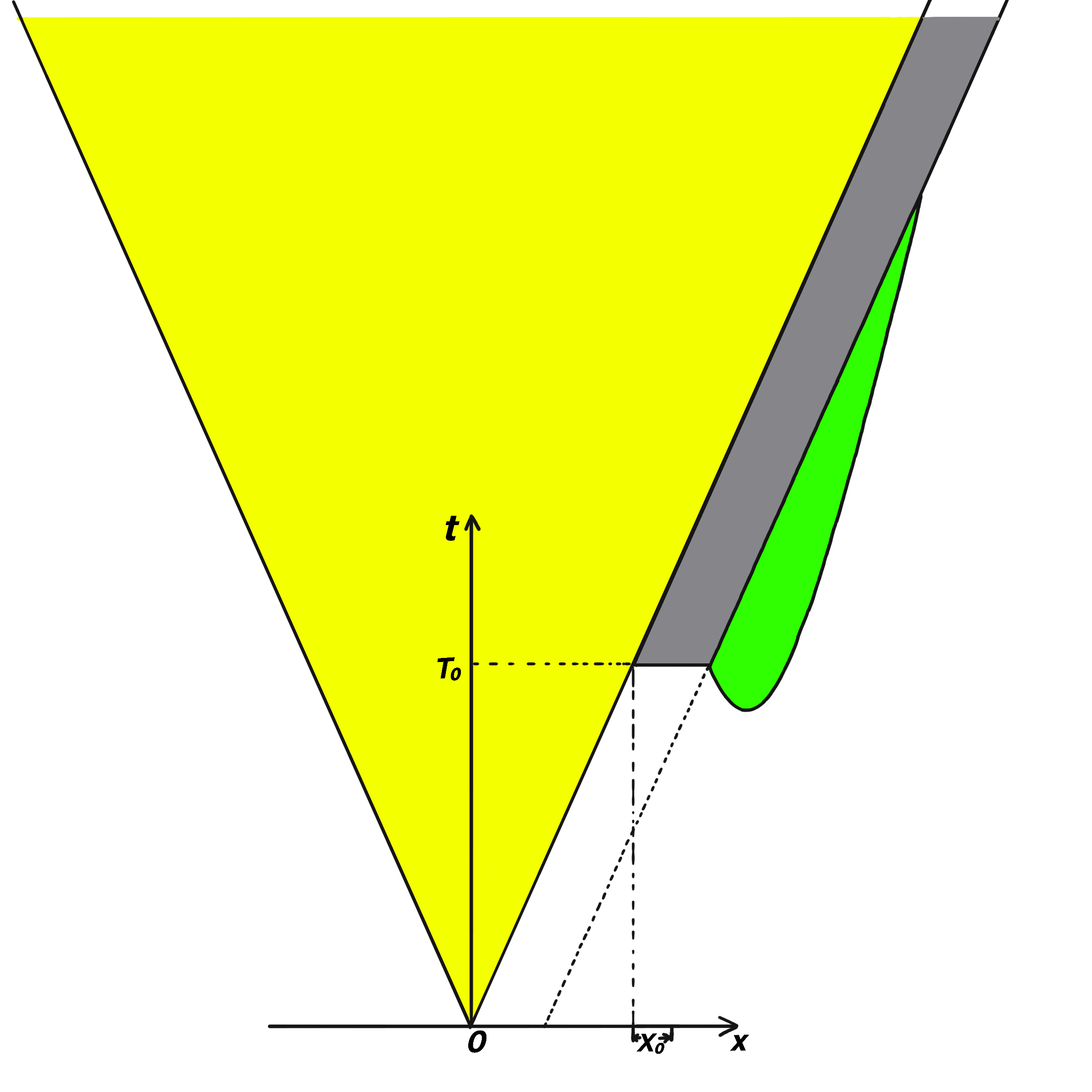} 
\caption{By repeatedly applying Lemma 6.1 on the boundary of the gray region, the green region is iteratively covered.}
\label{fig:1c}
\end{subfigure}
\hfill
\begin{subfigure}{0.35\textwidth}
\includegraphics[width=\textwidth]{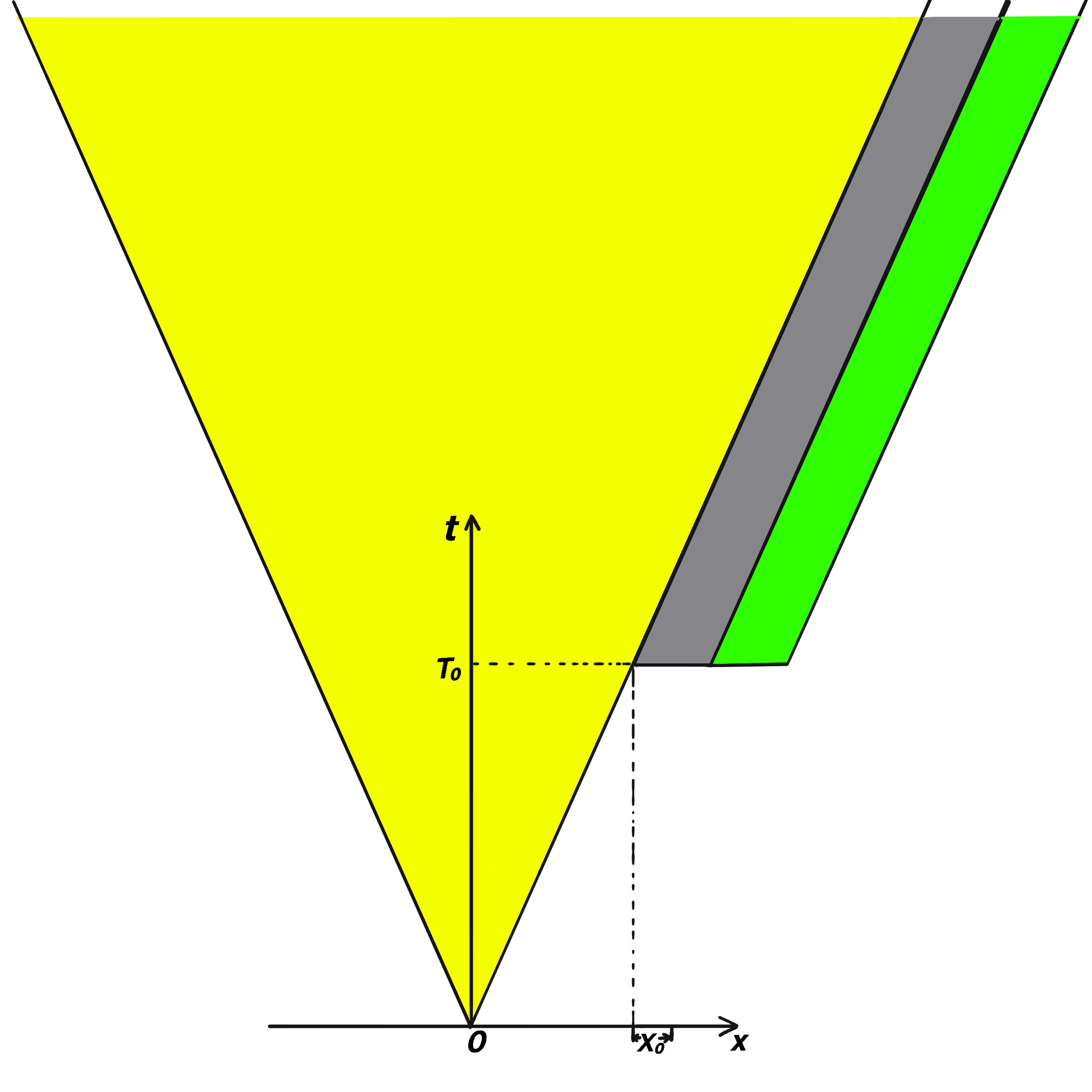} 
\caption{Similar to (b), Lemma 6.1 ensures $u\equiv 0$ in the expanded green area.}
\label{fig:1d}
\end{subfigure}

\caption{Using the local unique continuation property established in Lemma \ref{lemma6.1}, we divide the proof of Theorem \ref{th1} into the following steps. 
\textbf{Step 1.} Fix $(t_0,x_0)=(0,0)$. Since $u\equiv0$ in $Q_0(0,0)$ (the yellow region in Figure (a)), Lemma \ref{lemma6.1} immediately gives $u\equiv0$ in the gray region of Figure (a). 
\textbf{Step 2.} Let $(t_0,x_0)$ vary along the boundary $\partial Q_0(0,0)$. Repeating the argument of Step 1 at each such point shows that $u\equiv0$ in the gray region of Figure (b). 
\textbf{Step 3.} We now apply Lemma \ref{lemma6.1} iteratively, sliding the base point $(t_0,x_0)$ along the boundary of the zero region obtained in the previous step. This extends the vanishing property successively to the green domains shown in Figures (c) and (d). 
\textbf{Step 4.} Finally, the domain $\{(t,x)\mid t\ge T_0\}$ can be covered by the union of the zero regions constructed in the preceding iterations. Together with a backward uniqueness argument, this implies that $u\equiv0$ for all $t\ge0$, which completes the proof of Theorem \ref{th1}.}
\label{fig:all}
\end{figure}

\newpage

\begin{proof}[Proof of Theorem \ref{th1} from Lemma \ref{lemma6.1}]
We first consider the case where $u\equiv0$ in the cone $Q_0(0,0)$. By Lemma \ref{lemma6.1},
\begin{equation}\label{b1}
u\equiv0\quad\text{in}\quad\bigcup_{|y_1|^2=4c_3}\Bigl(Q_1(0,y_1)\setminus Q_0(0,0)\Bigr).
\end{equation}
In view of \eqref{b2}, relation \eqref{b1} implies
\begin{equation}\label{b3}
u(T_0,\cdot)\equiv0 \qquad\text{on } \bigl\{x\in\mathbb{R}^n\mid |x|\le\sqrt{\alpha}\,T_0+X_0\bigr\}.
\end{equation}

Now let the base point $(s_2,y_2)$ vary along the boundary $\partial Q_0(0,0)$. Since $Q_0(s_2,y_2)\subset Q_0(0,0)$, the same reasoning as for \eqref{b3} gives, for every $t\ge T_0$,
\begin{equation}\label{b4}
u(t,\cdot)\equiv0 \qquad\text{on } \bigl\{x\in\mathbb{R}^n\mid |x|\le\sqrt{\alpha}\,t+X_0\bigr\}.
\end{equation}

Next, fix a point $y_3\in\mathbb{R}^n$ with $|y_3|=X_0$ and apply Lemma \ref{lemma6.1} to the cone $Q_0(0,y_3)$. Set $y_4\=  y_3+(x_1-x_0)$. Observe that
$$
\partial Q_0(0,y_3)\cap Q_1(0,y_4)\subset\bigl\{(t,x)\mid |x|\le\sqrt{\alpha}\,t+X_0,\;t\ge T_0\bigr\}.
$$
By \eqref{b4} we therefore have
$$
u\equiv0\quad\text{in}\quad\bigcup_{|y_3-y_4|^2=4c_3}\Bigl(Q_1(0,y_4)\setminus Q_0(0,y_3)\Bigr),
$$
which in turn yields
$$
u(T_0,\cdot)\equiv0 \qquad\text{on } \bigl\{x\in\mathbb{R}^n\mid |x|\le\sqrt{\alpha}\,T_0+2X_0\bigr\}.
$$
Arguing as in the derivation of \eqref{b4}, we obtain for every $t\ge T_0$
$$
u(t,\cdot)\equiv0 \qquad\text{on } \bigl\{x\in\mathbb{R}^n\mid |x|\le\sqrt{\alpha}\,t+2X_0\bigr\}.
$$
Repeating this process inductively, we finally conclude that
$$
u\equiv0\quad\text{in } \{(t,x)\mid t\ge T_0\}.
$$
Combined with the backward uniqueness property of equation \eqref{main_equation}, this completes the proof of Theorem \ref{th1}.
\end{proof}

In the remainder of this section we prove Lemma \ref{lemma6.1}. Define the following Carleman weights:
\begin{gather}
\theta = e^{\ell},\qquad 
\ell = \lambda\psi,\qquad 
\psi = e^{\rho(t,x)},\label{weight-1}\\[4pt]
\rho(t,x)=\frac{\alpha}{2}(t-t_0)^2-|x-x_1|^2,\label{weight-2}
\end{gather}
where $\lambda$ is a (large) parameter that will be chosen later.

Put $\mu=0$ and $\gamma=1$ in Theorem \ref{th3} and denote $Q_0\= Q_1(t_0,x_1)\setminus Q_0(t_0,x_0)$. Then we obtain the following estimate.
\begin{theorem}\label{th4}
Let $\varrho\in C^2(\mathbb{R}^{1+n})$ be arbitrary, and let $\theta=e^{\ell},\;v=\theta w$ be as in \eqref{weight-1} and Lemma \ref{lem}. Then 
\begin{align}\label{b5}
&\mathbb{E}\int_{Q_0}\!\Bigl[\theta\bigl(-2\ell_t v_t+2\nabla\ell\cdot\nabla v+\Psi v\bigr)\bigl(dw_t-\Delta w\,dt\bigr)
+\nabla\!\cdot\!\mathcal{V}\,dt+d\mathcal{N}\Bigr]dx\notag\\[2pt]
&\ge\mathbb{E}\int_{Q_0}\!\Bigl\{2\lambda\psi\bigl(v_t,\nabla v^\top\bigr)\mathcal{M}(\varrho)\bigl(v_t,\nabla v^\top\bigr)^\top
+\bigl[\lambda^3\mathcal{D}_2+O(\lambda^2)\bigr]v^2\Bigr\}\,dt\,dx\notag\\
&\quad+\mathbb{E}\int_{Q_0}\!\bigl(-2\ell_t v_t+2\nabla\ell\cdot\nabla v+\Psi v\bigr)^2dt\,dx
+\mathbb{E}\int_{Q_0}\ell_t(dv_t)^2\,dx,
\end{align}
where $\mathcal{M}(\varrho)$ is defined in \eqref{ii-1},
$\Psi=\Delta\ell-\ell_{tt}+2\lambda\psi\varrho$, and
$$
\mathcal{D}_2=4\psi^3\varrho\bigl(\rho_t^2-|\nabla\rho|^2\bigr)
+2\psi^3\bigl(\rho_t,\nabla\rho^\top\bigr)\mathcal{M}(\varrho)\bigl(\rho_t,\nabla\rho^\top\bigr)^\top
+2\psi^3\bigl(\rho_t^2-|\nabla\rho|^2\bigr)^2 .
$$
\end{theorem}
To start the proof of Lemma \ref{lemma6.1} we also introduce a cut-off function. For a small constant $\varepsilon>0$ let $\chi:\mathbb{R}^{1+n}\to\mathbb{R}$ be a smooth function satisfying
\begin{equation}\label{b8}
\chi(t,x)=1\quad\text{on }\{(t,x)\mid\rho(t,x)>c_3+\varepsilon\},\qquad
\chi(t,x)=0\quad\text{on }\{(t,x)\mid\rho(t,x)<c_3\}.
\end{equation}
As in \eqref{w-e}, we have
\begin{align}\label{b9}
\theta\bigl(dw_t-\Delta w\,dt\bigr)
&=dv_t-\Delta v\,dt
+\Bigl[(\ell_t^2-|\nabla\ell|^2)v-(\ell_{tt}-\Delta\ell)v-2\ell_t v_t+2\nabla\ell\cdot\nabla v\Bigr]dt\notag\\
&=\bigl[(a_3-a_1\ell_t-\boldsymbol a_2\cdot\nabla\ell)v+a_1 v_t+\boldsymbol a_2\cdot\nabla v\bigr]dt\notag\\
&\quad+\theta\Bigl[-a_1\chi_t u-(\boldsymbol a_2\cdot\nabla\chi)u
+\chi_{tt}u+2\chi_t u_t-2\nabla\chi\cdot\nabla u+\Delta\chi u\Bigr]dt\\
&\quad+\bigl[(b_2-b_1\ell_t)v+b_1 v_t\bigr]dW(t) +\theta\bigl(-b_1\chi_t u\bigr)dW(t).\notag
\end{align}
Analogously to the preceding sections, Lemma \ref{lemma6.1} follows from the following Carleman estimate.

\begin{theorem}\label{th5}
Assume that $u,\;u_t,\;\nabla u$ vanish on $\partial Q_0(t_0,x_0)\cap Q_1(t_0,x_1)$.  
Then for every $\varepsilon>0$ there exist constants $\lambda_0>0$ and $C>0$ such that for any solution $w$ of \eqref{w-e} and all $\lambda\ge\lambda_0$,
\begin{align}\label{b10}
&\mathbb{E}\int_{Q_0}\theta^2\bigl(\lambda^3 w^2+\lambda w_t^2\bigr)\,dt\,dx\notag\\
&\le C\,\mathbb{E}\int_{Q_0}\theta^2\Bigl[-a_1\chi_t u-(\boldsymbol a_2\cdot\nabla\chi)u
+\chi_{tt}u+2\chi_t u_t-2\nabla\chi\cdot\nabla u+\Delta\chi u\Bigr]^2\,dt\,dx\notag\\
&\quad + C\,\mathbb{E}\int_{Q_0}\theta^2\bigl(b_1\chi_t u\bigr)^2\,dt\,dx .
\end{align}
\end{theorem}

\begin{proof}[Proof of Theorem \ref{th5}]
Take $\varrho=2$ in Theorem \ref{th4}. A direct computation gives
\begin{equation}\label{b11}
\mathcal{M}(2)=\begin{pmatrix}
\alpha-2 & 0\\
0 & 0
\end{pmatrix}.
\end{equation}
For $(t,x)\in Q_0$ we have
\begin{align}\label{b12}
\mathcal{D}_2
&\ge 32\psi^3\!\left[\frac{\alpha}{2}(t-t_0)^2-|x-x_1|^2\right]
+\psi^3\!\bigl[8\alpha^2-16\alpha+2(\alpha-2)\alpha^2\bigr](t-t_0)^2\notag\\
&\ge 32c_3\psi^3-16\alpha\psi^3(t-t_0)^2 .
\end{align}
As in \eqref{dv_t},
\begin{align}\label{b6}
\ell_t(dv_t)^2
= \lambda\phi_t\big[\bigl(b_2-b_1\ell_t\bigr)v+b_1 v_t+\theta\bigl(-b_1\chi_t u\bigr)\big]^2\,dt .
\end{align}
Hence
\begin{align}\label{b7}
&\lambda\phi_t\big[\bigl(b_2-b_1\ell_t\bigr)v+b_1 v_t+\theta\bigl(-b_1\chi_t u\bigr)\big]^2 \notag\\
&\ge \frac12\lambda\psi\alpha(t-t_0)c_1^2 v_t^2
+\frac12\lambda^3\psi^3\alpha^3(t-t_0)^3c_1^2 v^2
+\Bigl[\frac12\lambda\phi_t(b_2-b_1\lambda\phi_t)b_1 v^2\Bigr]_t\\
&\quad -O(\lambda^2)v^2-\theta^2\bigl(b_1\chi_t\bigr)^2 u^2.\notag
\end{align}
Combining \eqref{b11}, \eqref{b12} and \eqref{b7} with Theorem \ref{th4} yields
\begin{align}\label{b13}
&\mathbb{E}\int_{Q_0}\lambda^3\psi^3\!\Bigl[\tfrac12 \alpha^3(t-t_0)^3c_1^2+32c_3-16\alpha(t-t_0)^2\Bigr]v^2\,dt\,dx\notag\\
&+\mathbb{E}\int_{Q_0}\lambda\psi\!\Bigl[2(\alpha-2)+\tfrac12 \alpha(t-t_0)c_1^2\Bigr]v_t^2\,dt\,dx\\
&\le C\,\mathbb{E}\int_{Q_0}\theta^2\Bigl[-a_1\chi_t u-(\boldsymbol a_2\cdot\nabla\chi)u
+\chi_{tt}u+2\chi_t u_t-2\nabla\chi\cdot\nabla u+\Delta\chi u\Bigr]^2\,dt\,dx\notag\\
&\quad + C\,\mathbb{E}\int_{Q_0}\theta^2\bigl(b_1\chi_t u\bigr)^2\,dt\,dx .\notag
\end{align}
In the region $Q_0$ we have the lower bound $(t-t_0)\ge\sqrt{2c_3/\alpha}$. Recalling the definition of $c_3$ in \eqref{b14}, we obtain
$$
\frac12\lambda^3\alpha^3(t-t_0)^3c_1^2+32c_3-16\alpha(t-t_0)^2>0,
\qquad
2(\alpha-2)+\frac12 \alpha(t-t_0)c_1^2>0,
$$
for $\lambda$ large enough. Inserting these inequalities into \eqref{b13} gives the desired estimate \eqref{b10}.
\end{proof}

\begin{appendices}

\section{Nonuniqueness of soluitons to wave equations across the characteristic surface}\label{sec-nonuni}

In this section, we clarify that the Unique Continuation Property (UCP) generally fails across characteristic surfaces. While this fact is generally understood, we include a discussion here for completeness. Notably, the failure persists even for analytic coefficients, as demonstrated below.

\begin{lemma}\label{TheoremHor}\cite[Theorem 5.2.1]{Hormander}
Let $P(x,D) = \sum_{|\a| \le m} a_\a(x)D^\a$  be a differential operator of order $m$  defined in a neighborhood $G$ of a point $x_0\in \dbR^n$, and  with all coefficients $a_\a(\cd)$ are real  analytic  on $G$. Let $\phi $ be a real valued analytic function on $G$ such that $\n \varphi(x_0)\neq 0$ and $P_m(x, \n \varphi) \equiv  0 $ identically but $P_m^{(j)} (x, \n \varphi) \neq 0$ for some $j$ when $x=x_0$. Then there exists a neighborhood $G'$ of $x_0$ and a function $u\in C^m (G')$ which is analytic when $\varphi(x) \neq \varphi(x_0)$, such that $P(x,D) u  = 0$ and $\supp u =\{x\in G'|~ \varphi(x) \le \varphi(x_0)\}$.
\end{lemma}
\begin{remark}
To avoid ambiguity in Proposition \ref{TheoremHor}, we adopt the notation from H\"ormander's monograph (which may differ from our conventions). For the reader's convenience, we recall the following definitions:
\begin{itemize}
\item $D = i\partial$, where $i$ is the imaginary unit satisfying $i^{2} = -1$;
\item $P_{m}$ denotes the principal part of the operator $P$;
\item $P_{m}^{(j)}(x,\xi)$ stands for the partial derivative of $P_{m}$ with respect to $\xi$ corresponding to a multi-index $j \in \mathbb{N}^{n}$ with $|j| \le m$.
\end{itemize}
\end{remark}

For the deterministic wave equation, Lemma \ref{TheoremHor} implies that the Unique Continuation Property fails across characteristic surfaces. As an illustration, choose distinct points $(s_{0}, y_{0}) \neq (t_{0}, x_{0})$ with $s_{0} < t_{0}$ such that $\{(s_{0},y_{0})\} \cap Q_{0} = \emptyset$ and $|t_{0}-s_{0}| = |x_{0}-y_{0}|$, and define
\begin{equation}\label{rho-1}
\widetilde{\rho}(t,x) = |t-s_{0}| - |x-y_{0}|.
\end{equation}
\begin{corollary}\label{cor}
Assume that $a$ and $\bm{b}$ are real-analytic in $Q_{0}$. Then there exist a neighborhood $Q_{1}$ of $(s_{0}, y_{0})$ and a function $u \in C^{2}(Q_{0})$ satisfying
$$
u_{tt} - \Delta u - a u - \bm{b} \cdot \nabla u = 0,
$$
whose support is given by
\begin{equation}\label{supp}
\operatorname{supp} u = \{(t,x) \in Q_{0} \mid \widetilde{\rho}(t,x) \leq 0\}.
\end{equation}
\end{corollary}

\begin{proof}[Proof of Corollary \ref{cor}]
To apply Lemma \ref{TheoremHor}, we take $\varphi(t,x) = \widetilde{\rho}(t,x)$. 
Since $s_{0} \neq t_{0}$ and $y_{0} \neq x_{0}$, the derivatives of $\varphi$ at $(t_{0}, x_{0})$ are
\begin{equation}\label{1}
\bigl(\partial_{t}\varphi(t_{0}, x_{0}), \nabla\varphi(t_{0}, x_{0})\bigr)
= \left( \frac{t_{0}- s_{0}}{|t_{0} - s_{0}|},\; \frac{x_{0} - y_{0}}{|x_{0} - y_{0}|} \right) \neq (0,0).
\end{equation}
The principal symbol $P_{m}$ of the wave operator evaluated at this co‑vector vanishes:
\begin{equation}\label{2}
P_{m}\bigl(t, x, (\partial_{t}\varphi, \nabla\varphi)\bigr)
= 1 - 1 = 0.
\end{equation}
Moreover, its derivative with respect to the multi‑index $(1,0)$ is nonzero:
\begin{equation}\label{3}
P_{m}^{(1,0)}\Bigl(t_{0}, x_{0}, \bigl(\partial_{t}\varphi(t_{0},x_{0}), \nabla\varphi(t_{0}, x_{0})\bigr)\Bigr)
= \frac{2(t_{0} - s_{0})}{|t_{0} - s_{0}|} \neq 0.
\end{equation}
Conditions \eqref{1}, \eqref{2} and \eqref{3} together fulfill the hypotheses of Lemma \ref{TheoremHor}; hence \eqref{supp} follows.
\end{proof}

\section{The finite speed of propagation for solutions to the stochastic wave equation}

Let $K$ be a closed subset of $\mathbb{R}^n$. Denote by $\mathbf{d}_K(x)$ the distance from $x \in \mathbb{R}^n$ to $K$, i.e.,
$$
\mathbf{d}_K(x) \triangleq \inf\{|x - y| : y \in K\}.
$$
It is well known that $\mathbf{d}_K$ is Lipschitz continuous with constant $1$, i.e.,
$$
|\mathbf{d}_K(x) - \mathbf{d}_K(y)| \le |x - y|, \quad \forall x, y \in \mathbb{R}^n.
$$
Consequently, by Rademacher's theorem, $\mathbf{d}_K$ is differentiable almost everywhere, its gradient satisfies $|\nabla \mathbf{d}_K| \le 1$ almost everywhere, and $\mathbf{d}_K \in W^{1,\infty}_{\mathrm{loc}}(\mathbb{R}^n)$.

For each $r \ge 0$, define the closed $r$-neighborhood of $K$ by  
$$
K_r \triangleq \{x \in \mathbb{R}^n|~ \mathbf{d}_K(x) \le r\}.
$$
In this section, we borrow some ideas from \cite{VB} and   prove the following result.

\begin{proposition}\label{prop1}
Let $G = \mathbb{R}^n$ and $f = 0$ in equation \eqref{main_equation}.  
Let $
u \in L^2_{\mathbb{F}}(\Omega; C([0,T]; H^1_{\mathrm{loc}}(\mathbb{R}^n))) \cap L^2_{\mathbb{F}}(\Omega; C^1([0,T]; L^2_{\mathrm{loc}}(\mathbb{R}^n)))$  
be a solution of \eqref{main_equation} with initial data $u(0,\cdot) = u_0$, $u_t(0,\cdot) = u_1$.  
If 
\begin{equation}\label{A4}
\supp u_0\subset K, \q \supp u_1\subset K,
\end{equation}
then $\mathbb{P}$-almost surely, $\operatorname{supp} u(t,\cdot) \subset K_t$ for every $t \ge 0$.
\end{proposition}

\begin{proof}[Proof of Proposition \ref{prop1}] 
Choose a function $\rho \in C^1(\mathbb{R})$ satisfying  
\begin{equation}\label{A1}
\begin{cases}
	\rho(0) = 0, \\
	\rho(s) > 0 \quad \forall s > 0,\\
	\rho(s) = 0 \quad \forall s \le 0, \\
	\rho'(s) \ge 0 \quad \forall s \ge 0.
\end{cases}
\end{equation}
Under assumption \eqref{A4} and by standard regularity theory for \eqref{main_equation} (see, e.g., [QLXZ, Chapter 3]), the solution $u$ actually lies in  
$$
L^2_{\mathcal{F}}(\Omega; C([0,T]; H^1(\mathbb{R}^n))) \cap L^2_{\mathcal{F}}(\Omega; C^1([0,T]; L^2(\mathbb{R}^n))).
$$
Define the local energy functional $\mathcal{E}: [0,\infty) \to \mathbb{R}$ by  
$$
\mathcal{E}(t) \triangleq \frac12 \, \mathbb{E} \int_{\mathbb{R}^n} \rho\bigl(\mathbf{d}_K(x) - t\bigr)\Bigl(|\nabla u(t,x)|^2 + u_t^2(t,x) + u^2(t,x)\Bigr)\, dx.
$$
Condition \eqref{A4} ensures that $\mathcal{E}(0) = 0$. To prove the theorem, it suffices to show that $\mathcal{E}(t) = 0$ for all $t > 0$.

Applying It\^o's formula, we obtain  
\begin{equation}\label{A2}
\begin{aligned}
	d\mathcal{E}(t) &= -\frac12 \, \mathbb{E} \int_{\mathbb{R}^n} \rho'\bigl(\mathbf{d}_K(x)-t\bigr)
	\Bigl(|\nabla u|^2 + u_t^2 + u^2\Bigr)\, dx\, dt \\
	&\quad + \mathbb{E} \int_{\mathbb{R}^n} \rho\bigl(\mathbf{d}_K(x)-t\bigr)
	\Bigl[\nabla u \cdot \nabla u_t\, dt + u_t \, du_t + \frac12 (du_t)^2 + u u_t\, dt\Bigr] \, dx.
\end{aligned}
\end{equation}
Using the equation \eqref{main_equation} and integrating by parts in space yields  
\begin{equation}\label{A3}
\begin{aligned}
	\mathcal{E}(t) &= -\frac12 \, \mathbb{E} \int_0^t \int_{\mathbb{R}^n} \rho'\bigl(\mathbf{d}_K(x)-\tau\bigr)
	\Bigl[|\nabla u|^2 + u_t^2 + u^2 + 2\nabla u \cdot \nabla \mathbf{d}_K(x)\, u_t \Bigr] \, dx\, d\tau \\
	&\quad + \mathbb{E} \int_0^t \int_{\mathbb{R}^n} \rho\bigl(\mathbf{d}_K(x)-\tau\bigr)
	\Bigl[ u_t(a_1 u_t - \mathbf{a}_2\!\cdot\!\nabla u + a_3 u) + \frac12 (b_1 u_t + b_2 u)^2 + u u_t\Bigr] \, dx\, d\tau .
\end{aligned}
\end{equation}
Recalling that $|\nabla \mathbf{d}_K| \le 1$ almost everywhere, by Chauchy-Schwartz inequality, we get from \eqref{A3} that  
$$
\mathcal{E}(t) \le C \int_0^t \mathcal{E}(\tau)\, d\tau
$$
for some constant $C > 0$ independent of $t$.  
Since $\mathcal{E}(0)=0$, Gronwall's inequality implies $\mathcal{E}(t) = 0$ for all $t \ge 0$.  
This completes the proof. 
\end{proof}

\end{appendices}


\begin{thebibliography}{99}
\bibitem{VB} V. Barbu and M. R\"ockner, \it 
The finite speed of propagation for solutions to nonlinear stochastic wave equations driven by multiplicative noise. \sl 
J. Differential Equations \rm 255 (2013), 560--571.

\bibitem{APC} 
A.-P. Calder\'on, \it Uniqueness in the Cauchy problem for partial differential equations. \sl Amer. J. Math. \rm 80 (1958), 16--36.

\bibitem{XFQLXZ}  
X. Fu, Q. L\"u and X. Zhang, \it 
Carleman estimates for second order partial differential operators and applications.
A unified approach.  \sl SpringerBriefs in Mathematics. BCAM SpringerBriefs. Springer, Cham, \rm 2019.




\bibitem{Hormander}   L. H\"ormander, \sl Linear partial differential operators.  \rm Springer-Verlag, Berlin-New York, 1976.  

\bibitem{Hor1}  
L. H\"ormander,  \sl The analysis of linear partial differential operators. I. Distribution theory and Fourier analysis. Second edition.  \rm  Springer-Verlag, Berlin,  1990.  
 
\bibitem{AISK} 
A. Ionescu and S.  Klainerman,  \it 
Uniqueness results for ill-posed characteristic problems in curved space-times. \sl 
Comm. Math. Phys. \rm  285 (2009), 873--900.

\bibitem{CEKARCDAS}  
C. E. Kenig, A. Ruiz  and  C. D. Sogge, \it 
Uniform Sobolev inequalities and unique continuation for second order constant coefficient differential operators. \sl 
Duke Math. J. \rm 55 (1987),  329--347.

\bibitem{MVK} 
M. V. Klibanov  and   A. Timonov, \sl
Carleman estimates for coefficient inverse problems and numerical applications. \rm  VSP, Utrecht, 2004.  

\bibitem{PK} 
P. Kotelenez,  \sl   Stochastic ordinary and stochastic partial differential equations. Transition from microscopic to macroscopic equations.   \rm  Springer, New York, 2008.  
 
\bibitem{NL} 
N. Lerner, \sl  Carleman inequalities. An introduction and more. \rm Springer, Cham,  2019.  

\bibitem{NL1}
N. Lerner,  \it 
Unique continuation through transversal characteristic hypersurfaces.  \sl 
J. Anal. Math. \rm 138 (2019),  135--156.

\bibitem{ZLQL}
Z. Liao and Q. L\"u, \it 
Strong unique continuation property for stochastic parabolic equations. \sl 
Sci. China Math. \rm 68 (2025),  1443--1460.

\bibitem{QL1} 
Q. L\"u, \it 
Observability estimate and state observation problems for stochastic hyperbolic equations. \sl 
Inverse Problems \rm 29 (2013),  095011.

\bibitem{Lu2023} 
Q. L\"u, \it Control theory of stochastic distributed parameter systems: recent progress and open problems. \sl ICM--International Congress of Mathematicians. Vol. 7. Sections 15--20, 5314--5338.
EMS Press, Berlin 2023.

\bibitem{QLZY1} 
Q. L\"u and Z. Yin, \it 
Local state observation for stochastic hyperbolic equations. \sl 
ESAIM Control Optim. Calc. Var. \rm 26 (2020), Paper No. 79.

\bibitem{QLZY} 
Q. L\"u and Z. Yin, \it  Unique continuation for stochastic heat equations. \sl ESAIM Control Optim. Calc. Var. \rm 21 (2015),  378--398.  

\bibitem{Lu2015} Q. L\"u and X. Zhang, \it Global uniqueness for an inverse stochastic hyperbolic problem with three unknowns. \sl
Comm. Pure Appl. Math.  \rm  68 (2015), 948--963.


\bibitem{QLXZ} 
Q. L\"u and X. Zhang,  \sl 
Mathematical control theory for stochastic partial differential equations.  \rm Springer, Cham,  2021.

\bibitem{JLR}  
J. L Rousseau, G. Lebeau and L. Robbiano,  \sl  Elliptic Carleman estimates and applications to stabilization and controllability. Vol. I. Dirichlet boundary conditions on Euclidean space.  \rm   Birkh\"auser/Springer, Cham,  2022. 

\bibitem{XZ1} 
X. Zhang, \it 
Carleman and observability estimates for stochastic wave equations. \sl 
SIAM J. Math. Anal. \rm 40 (2008),  851--868.


\bibitem{XZ} 
X. Zhang, \it  Unique continuation for stochastic parabolic equations.  \sl Differential Integral Equations \rm 21 (2008),  81--93.

\bibitem{Zhang2010} 
X. Zhang, \it  
A unified controllability/observability theory for some stochastic and deterministic partial differential equations. \sl Proceedings of the International Congress of Mathematicians. Volume IV, \rm 3008--3034.
Hindustan Book Agency, New Delhi,  2010.

\bibitem{CZ} C. Zuily,  \sl  Uniqueness and Non-uniqueness in the Cauchy problem.  \rm Birkh\"auser Boston Inc., Boston, MA, 1983.
 

\end{thebibliography}
\end{document}